\documentclass[a4paper,11pt]{article}

\usepackage{amsmath,amsthm,amsfonts,amssymb,color}
\usepackage[english]{babel}
\usepackage{authblk}

\newtheorem{theorem}{Theorem}[section]
\newtheorem{corollary}[theorem]{Corollary}

\newtheorem{proposition}[theorem]{Proposition}
\newtheorem{lemma}[theorem]{Lemma}

\newtheorem{remark}[theorem]{Remark}

\newcommand{\E}[1]{\mathbb{E} \Big[ #1 \Big]}
\newcommand{\Es}[1]{\mathbb{E} [ #1 ]}

\def\EE{\mathbb{E}}
\def\H{\mathcal{H}}

\def\R{\mathbb{R}}
\def\N{\mathbb{N}}
\def\C{\mathcal{C}}
\def\F{\mathcal{F}}
\def\P{\mathbb{P}}

\begin{document}

\title{SPDEs with linear multiplicative fractional noise: continuity in law with respect to the Hurst
index}

\author[1]{Luca M. Giordano}
\author[2]{Maria Jolis}
\author[3]{Llu\'is Quer-Sardanyons \thanks{Corresponding author}}

\affil[1]{Department of Mathematics, University of Milano, Via C.
Saldini 50, 20133 Milano, Italy, and Departament de Matem\`atiques,
Universitat Aut\`{o}noma de Barcelona, 08193 Bellaterra, Catalonia,
Spain. E-mail address: luca.giordano@unimi.it}

\affil[2,3]{Departament de Matem\`atiques, Universitat Aut\`{o}noma
de Barcelona, 08193 Bellaterra, Catalonia, Spain. E-mail addresses:
mjolis@mat.uab.cat, quer@mat.uab.cat.}

\date{\today}
\maketitle

\begin{abstract}
In this article, we consider the one-dimensional stochastic wave and
heat equations driven by a linear multiplicative Gaussian noise
which is white in time and behaves in space like a fractional
Brownian motion with Hurst index $H \in (\frac14,1)$. We prove that
the solution of each of the above equations is continuous in terms
of the index $H$, with respect to the convergence in law in the
space of continuous functions. The proof is based on a tightness
criterion on the plane and  Malliavin calculus techniques in order
to identify the limit law.
\end{abstract}

\medskip

\noindent {\em MSC 2010:} 60B10; 60H07; 60H15

\smallskip

\noindent {\em Keywords:}  fractional noise; stochastic heat
equation; stochastic wave equation; weak convergence; Wiener chaos
expansion


\section{Introduction}

In this article, we consider the Hyperbolic Anderson Model
\begin{equation}
\left\{\begin{array}{l} \displaystyle \frac{\partial^2
u^H}{\partial t^2}(t,x)  =  \displaystyle \frac{\partial^2
u^H}{\partial x^2}(t,x)+u^H(t,x)\dot{W}^H(t,x),
\quad (t,x)\in \R_+\times \R, \\[2ex]
\displaystyle u^H(0,x)  =  \eta,\; x\in \R, \\[1ex]
\displaystyle \frac{\partial u^H}{\partial t}(0,x)  =  0, \;
x\in \R,
\end{array}\right. \label{wave} \tag{SWE}
\end{equation}
and the Parabolic Anderson Model
\begin{equation}
\left\{\begin{array}{l}
\displaystyle \frac{\partial u^H}{\partial t}(t,x)  = \displaystyle  \frac{1}{2}\, \frac{\partial^2 u^H}
{\partial x^2}(t,x) + u^H(t,x)\dot{W}^H(t,x), \quad (t,x)\in \R_+\times \R, \\[2ex]
\displaystyle u^H(0,x)  =  \eta, \; x\in \R.
\end{array}\right. \label{heat} \tag{SHE}
\end{equation}
The initial condition $\eta\in \R$ is assumed to be constant. The
random perturbation $\dot{W}^H$ is a Gaussian noise which is white
in time and behaves in space like a fractional Brownian motion with
Hurst index $H\in (\frac14,1)$. More precisely, it is given by a
family of centered Gaussian random variables $W^H=\{W^H(\varphi),\,
\varphi \in \C^\infty_0(\R_+\times \R)\}$, indexed in the space of
$\C^\infty$ functions with compact support on $\R_+\times \R$, with
the following covariance structure:
\begin{equation*}
\EE \left[W^H(\varphi)W^H(\psi)\right] = \int_{0}^{\infty}\int_{\R}
\F \varphi(t,\cdot)(\xi) \overline{\F \psi(t,\cdot)(\xi)}\mu_H(d\xi)
dt,
\end{equation*}
for any $\varphi,\psi \in \C^\infty_0(\R_+\times \R)$, where the
measure $\mu_H$ is given by $\mu_H(d\xi)=c_H\, |\xi|^{1-2H} d\xi$,
with
\begin{equation}\label{eq:103}
c_H=\frac{\Gamma(2H+1) \sin(\pi H)}{2\pi}.
\end{equation}
We denote by $\F$ the Fourier transform in the space variable, which is defined by
\[
    \F f (\xi)=\int_\R e^{-i \xi x}f(x)dx, \quad f\in L^1(\R).
\]

The solutions of \eqref{wave} and \eqref{heat} are understood in the
mild It\^o sense, as follows. We fix a time horizon $T>0$ and we
denote by $\{\F_t^H,\, t\geq 0\}$ the filtration generated by the
noise $W^H$ (conveniently completed). Then, we say that an adapted
and jointly measurable random field $u^H=\{u^H(t,x),\,
(t,x)\in[0,T]\times \R\}$ solves \eqref{wave} (resp. \eqref{heat})
if it holds, for all $(t,x)\in [0,T]\times \R$:
\begin{equation}
u^H(t,x)=\eta + \int_0^t\int_\R G_{t-s}(x-y)u^H(s,y) W^H(ds,dy),\quad \mathbb{P}\text{-a.s.}
\label{eq:222}
\end{equation}
where $G$ is the fundamental solution of the
wave  (resp. heat) equation in $\R$. We recall that
\begin{equation}
G_t(x)=\begin{cases}
\frac{1}{2}1_{\{|x|<t\}}(x), & \text{wave equation,} \\
\\
\dfrac{1}{(2\pi t)^{\frac{1}{2}}}\exp\Big(-\dfrac{|x|^2}{2t}
\Big), & \text{heat equation}. \\
\end{cases}
\label{eq:3}
\end{equation}
The stochastic integral appearing in \eqref{eq:222} is understood in the It\^o
sense and will be described in detail in Section \ref{subsec: stochastic integral SPDE}.

\smallskip

In this paper, we are interested in studying the continuity in law,
in the space of continuous functions, of the solutions to
\eqref{wave} and \eqref{heat} with respect to the Hurst index $H$.
More precisely, the main result of the paper is the following:
\begin{theorem}
 \label{thm:main}
Let $T>0$. Let $H_0\in
(\frac{1}{4},1)$ and $\{H_n,\,n\geq 1\}\subset (\frac 14,1)$ be any
sequence converging to $H_0$. Then,  $u^{H_n}$ converges to $u^{H_0}$, as $n\to\infty$, in law
in the space $\C([0,T]\times \R))$ of continuous functions, endowed with the metric of uniform convergence on compact
sets.
 \end{theorem}

We point out that we restrict to Hurst indices greater than $\frac14$. This is due to the fact
that, as proved in \cite[Prop. 3.7]{intermittency}, $H>\frac14$ is also a necessary condition in order to
have a solution to \eqref{wave} and \eqref{heat}.

The above theorem can be considered a continuation of the results
obtained by the authors in \cite{paper1} (see Theorem 4.2 therein),
where the same kind of problem has been addressed for
one-dimensional quasi-linear stochastic wave and heat equations with
an additive fractional noise as the one described above. The proof
of the latter result, which is indeed valid for any $H_0\in (0,1)$,
is based on the fact that the solution of the underlying SPDE can be
represented as the image of the stochastic convolution through a
continuous functional on the space $\C([0,T]\times \R))$. In the
present paper, this technique cannot be applied anymore because of
the structure of the linear multiplicative noise. Instead, we
consider the following strategy.

First, we prove that the sequence of probability measures induced by
$\{u^{H_n},\, n\geq 1\}$ is tight in the space $\C([0,T]\times \R)$
(see Section \ref{sec:tightness}). Here, we split the proof taking
into account that the sequence of Hurst indices is contained in
$(\frac14,\frac12]$ or $[\frac12,1)$, for the definition and
properties of the stochastic integral in \eqref{eq:222} differ
significantly between those two cases. Indeed, the main difficulty
here is concentrated in the {\it{rough}} case, where we carefully
extend some moment estimates appearing in \cite{intermittency} in
order to make them uniform with respect to $H\in (\frac14,\frac12)$.

Secondly, in order to identify the limit law, we prove the
convergence of the corresponding finite dimensional distributions
(see Section \ref{subsec: limit identification LMC}). The main
problem here comes from the fact that the solution $u^H$ is not a
Gaussian process, and so identifying its covariance structure is not
enough to characterize its law. However, thanks to a spectral
representation in law of our noise $W^H$ in terms of a
complex-valued Gaussian measure (extending of some classical results
in \cite{Pipiras 2000}), we are able to define the whole family of
noises $\{W^H,\, H\in (0,1)\}$ in a single probability space and
then check that, for any fixed $(t,x)\in [0,T]\times \R$,
$u^{H_n}(t,x)$ converges to $u^{H_0}(t,x)$ in $L^2(\Omega)$. For
this, we will use techniques of the Malliavin calculus, precisely
the Wiener chaos expansion of the mild Skorohod solutions of
\eqref{wave} and \eqref{heat}. In the process of applying this
methodology, we provide three preliminary results which have their
own interest and turn out to be crucial in our main result's proof:

\begin{itemize}\itemsep0cm
\item[(i)] For any $H\in (0,1)$, we prove that any multiple Wiener
integral with respect to $W^H$ admits a representation as a multiple
Wiener integral with respect to the above-mentioned complex-valued
Gaussian measure (see Theorem \ref{teo: representation multiple
integral}).

\item[(ii)] For any $H\in (\frac 14,1)$, we prove an equivalence result
between It\^o and Skorohod stochastic integrals with respect to
$W^H$ (see Theorem \ref{teo: equivalence ito-skorohod}). This result
has already been proved in \cite[Thm. 4.2]{intermittency} for the
case $H<\frac12$, and we extend it to $H\geq \frac12$. We point out
that the latter case, in which the noise is more regular, entails
some extra difficulties due to the fact that the underlying Hilbert
space associated to the noise's covariance contains distributions.
An important consequence of Theorem \ref{teo: equivalence
ito-skorohod} is that mild It\^o and Skorohod solutions to
\eqref{wave} (resp. \eqref{heat}) coincide, and the corresponding
Picard iteration scheme admits a (finite) Wiener chaos
decomposition.

\item[(iii)] In the setting $H\in (\frac14,\frac12)$, we prove a
Sobolev embedding-type result for the norms of the Banach space on
which we define our solutions (see Lemma \ref{lemma: sobolev
embedding}). This result is similar to classical embedding results,
e.g. the ones appearing in \cite{hitch}, but takes into account the
different nature of the Sobolev norm in our setting.
\end{itemize}

The above strategy will be made clearer in Section \ref{subsec:
limit identification LMC} below, but let us remark at this point
that the main strategy in this part of the paper does not require a
separate analysis for the cases $H<\frac12$ and $H\geq \frac12$.
Furthermore, the methodology used in both results on tightness and
the limit identification cover equations \eqref{wave} and
\eqref{heat} at the same time.

In the case of the stochastic heat equation \eqref{heat} with
$H>\frac12$, the result in Theorem \ref{thm:main} is a particular
case of \cite[Thm. 1]{Bezdek}, where the author considers a general
non-linear coefficient $\sigma(u^H(t,x))$ in front of the noise. We
believe that such diffusion coefficient could be also considered in
the case of the wave equation with $H>\frac12$, but we have chosen
to stick to the linear multiplicative noise in order to find a
unified result that covers also the case $H<\frac12$, which is more
mathematically demanding.

Concerning other related results, we point out the recent article
\cite{Koch}, in which the authors prove strong regularity properties
in $H\in (0,1)$ of the  Mandelbrot-van Ness representation of the
fractional Brownian motion. As a consequence, it is proved that the
solution of a scalar stochastic differential equation driven by the
fractional Brownian motion is differentiable with respect to the
Hurst parameter.

Finally, we also mention that continuity in law with respect to the
Hurst index has been focused in other type of contexts beyond
stochastic equations. For instance, in the series of papers
\cite{JV2,JV3,JV4}, the authors study weak continuity with respect
to $H$ for different types of integrals with respect to  fractional
Brownian motion. In \cite{JV1,Xiao}, the same kind of continuity
property has been tackled for the local time of the fractional
Brownian motion and other Gaussian fields. Eventually, in the recent
paper \cite{Ait}, the continuity property has been shown for
additive functionals of the sub-fractional Brownian motion.

The paper is organized as follows. In Section \ref{sec: prel}, we
give some preliminary tools that will be needed throughout the
paper. Namely, we introduce the basic elements of the Malliavin
calculus, we provide a new integral representation for the multiple
Wiener integral with respect to $W^H$, we recall the construction of
the stochastic integral with respect to $W^H$ and, finally, we
report about the existing well-posedness results for equations
\eqref{wave} and \eqref{heat}. Section \ref{sec:tightness} is
devoted to prove the tightness property of the family of laws
induced by the solution $u^H$, $H\in (\frac14,1)$. In Section
\ref{subsec: limit identification LMC}, we deal with the limit
identification, which allows us to conclude the proof of Theorem
\ref{thm:main}. In the Appendix, we collect some technical results
and a tightness criterion that are used in the paper.


\section{Preliminaries}
\label{sec: prel}

\subsection{Malliavin calculus}
\label{sec:Malliavin}

In this section, we recall some elements of Malliavin
calculus and a useful result of \cite{paper greeks}. We refer the
reader to \cite{nualart book} for more details. We will work in the
Gaussian space determined by the noise $W^H$, which is defined as follows.

Let $\langle\varphi,\psi\rangle_H:=\EE
\left[W^H(\varphi)W^H(\psi)\right]$ and define $\H_H$ as the
completion of $\C^\infty_0(\R_+\times \R)$ with respect to the inner
product $\langle\cdot,\cdot\rangle_H$. Then $\H_H$ defines a Hilbert
space and it is well-known that, if $H\leq \frac12$, it is a space
of functions, while for $H>\frac12$ it contains distributions (see
\cite[Thm. 4.3]{Basse} and \cite[Prop. 4.2]{maria space H^H}). Then,
$\{W^H(\varphi),\, \varphi \in \C^\infty_0(\R_+\times \R)\}$ can be
extended to a family of Gaussian random variables indexed on the
space $\H_H$, which we denote again by $W^H=\{W^H(\varphi),\,
\varphi\in \H_H\}$. This family defines an isonormal Gaussian
process on the Hilbert space $\H_H$: for any $\varphi\in \H_H$,
$W^H(\varphi)$ is a centered Gaussian variable and
\[
\EE\left[W^H(\varphi),W^H(\psi)\right]=\langle\varphi,\psi\rangle_H,
\quad \varphi,\psi \in \H_H.
\]
Let $\mathcal{G}^H$ be the $\sigma$-algebra generated by
$\{W^H(\varphi), \, \varphi \in \mathcal{H}_H\}$. Then, any
$\mathcal{G}^H$-measurable random variable $F\in L^2(\Omega)$ admits
the representation
\begin{equation}
F=\sum_{n\geq 0} J_n^H F,
\label{eq: Wiener Chaos basic}
\end{equation}
where $J_n^H F$ is the projection of $F$ on the $n$-th Wiener chaos
space $\mathbb{H}_{H,n}$, for $n\geq 1$, and $J^H_0F=\EE[F]$.

We denote by $I_n^H$ the multiple Wiener integral of order $n$ with
respect to $W^H$, which defines a linear and continuous operator
from $\H_H^{\otimes n}$ onto $\mathbb{H}_{H,n}$. We briefly recall
the construction of $I_n^H$, since we will use some of its steps in the
sequel. Let $\{e_k,\, k\geq 1\}$ be an orthonormal basis of
$\mathcal{H}_H$ and consider an {\it{elementary element}} of
$\mathcal{H}_H^{\otimes n}$ of the form
\begin{equation}
\varphi =  e_{i_1}\hat{\otimes} \cdots \hat{\otimes}
\, e_{i_n}, \label{eq:7}
\end{equation}
where $\hat{\otimes}$ denotes the symmetrized tensor product, for
some $i_1,\dots,i_n\geq 1$.
Recall that the set of finite linear combinations of elementary
elements is dense in $\mathcal{H}_H^{\otimes n}$. An elementary
element of the form \eqref{eq:7} can be more conveniently written as
\begin{equation}
    \label{eq: elementary function}
    \varphi = e_{j_1}^{\otimes {k_1}}\hat{\otimes} \dots \hat{\otimes}\,
    e_{j_m}^{\otimes {k_m}},
\end{equation}
where all $j_1,\dots j_m\geq 1$ are different and
$k_1+\cdots+k_m=n$. The $n$-th order multiple Wiener integral of
$\varphi$ is defined as follows:
\begin{equation}
    \label{eq: multiple wiener integral}
    I^H_n(\varphi) = P_{k_1}\big(W^H(e_{j_1})\big)\cdots
    P_{k_m}\big(W^H(e_{j_m})\big),
\end{equation}
where we denote by $P_{k}$ the normalized $k$-th Hermite polynomial.
The multiple Wiener integral is then extended by linearity to all
finite linear combinations of elementary elements, and finally
extended to the whole space $\H_H$ by density.

We also remind that any element in the $n$-th chaos
$\mathbb{H}_{H,n}$ can be represented as $I_n^H(f)$, for some $f\in
\H_H^{\otimes n}$. Hence, representation \eqref{eq: Wiener Chaos
basic} can be written as follows:
\[
F= \EE[F] + \sum_{n\geq 1} I_n^H(f_n),
\]
where $f_n\in \H_H^{\otimes n}$, for all $n\geq 1$. We recall that,
for any $f\in\mathcal{H}_H^{\otimes n},$
$$\EE\left[|I^H_n(f)|^2\right]=\EE\left[|I^H_n(\tilde{f})|^2\right]=n! \,
\|\tilde{f}\|^2_{\mathcal{H}_H^{\otimes n}},$$ where $\tilde{f}$
stands for the symmetrization of $f$. We also remind that, for a
general element $f$ of $\mathcal{H}_H^{\otimes n}$, the norm
$\|f\|_{\mathcal{H}_H^{\otimes n}}$ is given by
    \begin{equation*}
    \|f\|^2_{\mathcal{H}_H^{\otimes n}}=\int_{\R_+^n}\int_{\R^n} |\F f(t_1,\cdot,t_2,\cdot,\dots,t_n,\cdot)(\xi_1,\dots,\xi_n)|^2 \mu(d\xi_1)\cdots \mu(d\xi_n)dt_1\cdots dt_n.
    \end{equation*}
    Here, we still denoted by $\F$ the Fourier transform on the space of tempered
    distributions in $\R^n$.

\medskip

Let $A\in\mathcal{B}([0,\infty))$. We define, for every $f\in
\mathcal{H}_H^{\otimes n}$, the element $f1^{\otimes n}_A\in
\mathcal{H}_H^{\otimes n}$ in the following way: if $f$ is a
function, we define it obviously as the function $f1^{\otimes n}_A$.
If $f$ is a general element of $\mathcal{H}_H^{\otimes n}$, we take
any sequence $\{f_k,\, k\geq 1\}$ of functions in
$\mathcal{H}_H^{\otimes n}$ such that $f_k \to f$ in
$\mathcal{H}_H^{\otimes n}$, as $k\to\infty$, and we set
$$f1_A^{\otimes n}:=\lim_{k\to\infty} f_k 1_A^{\otimes n}.$$
This limit exists; indeed, we have that $\{f_k,\, k\geq 1\}$ is
Cauchy in $\mathcal{H}_H^{\otimes n}$ and
$$\|f_k1_A^{\otimes n}-f_\ell1_A^{\otimes n}\|_{\mathcal{H}_H^{\otimes n}}\leq
\|f_k-f_\ell\|_{\mathcal{H}_H^{\otimes n}},$$ which implies that
$\{f_k1_A^{\otimes n}, \, k\geq 1\}$ is also a Cauchy sequence in
$\mathcal{H}_H^{\otimes n}$. The limit clearly does not depend on
the chosen approximating sequence. On the other hand, we define
the $\sigma$-field
$$\mathcal F_A^H=\sigma\{W^H(1_{D} \varphi),\, D\in\mathcal
B_0(\R_+), \, D\subset A,\, \varphi \in \mathcal{C}^\infty_0(\R)\}\vee \mathcal N,$$
where $\mathcal N$
are the null sets of $\mathcal{F}$ and $\mathcal{B}_0(\R_+)$ are the
bounded Borel sets of $\R_+$.

We have the following result:
\begin{lemma}
\label{lemma: A.1} Let $F\in L^2(\Omega)$ with Wiener chaos
expansion given by $F= \EE[F] + \sum_{n\geq 1} I_n^H(f_n)$, where
$f_n\in \mathcal{H}_H^{\otimes n}$ are symmetric, and let $A\in
\mathcal{B}([0,\infty))$. Then, it holds
$$\EE\left[F|\F_A^H\right]=\sum_{n\geq 0} I_n^H\left(f_n1_{A}^{\otimes n}\right).$$
\end{lemma}
\begin{proof}
The proof follows exactly as that of \cite[Lem. A.1]{intermittency}.
We only need to observe that, if $h\in\mathcal{H}_H^{\otimes n}$ is
symmetric, it can be written as the limit of a sequence of symmetric
functions, which in turn can be written as the limit of finite
linear combinations of functions of the type $f^{\otimes n}$, where
$f\in \mathcal{H}_H$ and $\|f\|_{\mathcal{H}_H}=1$.
\end{proof}

\medskip

Let us now introduce the Malliavin derivative operator and the
Skorohod integral. Let $\mathcal{S}$ be the class of random
variables $F$ of the form
    \begin{equation*}
        F=f(W^H(\varphi_1),\dots,W^H(\varphi_n)),
    \end{equation*}
    where $f\in\C^\infty_b(\R^n)$ and $\varphi_j\in \mathcal{H}_H$, for every $j=1,\dots,n$. For any $F\in \mathcal{S}$, we define the \textit{Malliavin derivative} of $F$ as the $\mathcal{H}_H$-valued random variable $DF$ given by
    \begin{equation*}
       DF=\sum_{j=1}^{n} \frac{\partial f}{\partial x_j} (W^H(\varphi_1),\dots,W^H(\varphi_n))\varphi_j.
    \end{equation*}
If we endow $\mathcal{S}$ with the norm
$\|F\|_{\mathbb{D}^{1,2}}:=\EE\left[|F|^2\right]^{\frac 12}+
\EE\Big[ \|DF\|^2_{\mathcal{H}_H}\Big]^{\frac 12}$, it turns out
that the operator $D$ can be extended to the completion of
$\mathcal{S}$ with respect to $\|\cdot\|_{\mathbb{D}^{1,2}}$, which
we will denote by $\mathbb{D}^{1,2}$. We define now the
\textit{divergence operator} $\delta$, which is the adjoint of $D$.
The divergence operator is defined on its domain
$\text{Dom}(\delta)$, which is the space of $\mathcal{H}_H$-valued
random variables such that $u\in L^2(\Omega;\mathcal{H}_H)$ and
$$\left|\EE\left[\langle DF, u\rangle_{H}\right]\right|\leq c \EE\left[|F|^2\right]^{\frac 12}, \quad \text{ for all } F\in \mathbb{D}^{1,2},$$
where the constant $c$ depends on $u$.  Being the adjoint of $D$,
the divergence operator $\delta(u)$ is defined for any $u \in
\text{Dom}(\delta)$ by the duality relation, holding for every $F\in
\mathbb{D}^{1,2}$:
    $$\E{ \langle DF, u\rangle_{H} }=\Es{F\delta(u)}.$$
From the duality relation one can deduce that $\Es{\delta(u)}=0$,
for every $u\in \text{Dom}(\delta)$. For any $u\in
\text{Dom}(\delta)$, $\delta(u)$ is called the \textit{Skorohod
integral} of $u$ and is denoted by
    $$\int_{0}^{\infty}\int_{\R} u(t,x)W^H(\delta t, \delta x):=\delta(u).$$
We will need the following two results involving the Skorohod
integral (cf. Propositions 1.3.3 and 1.3.6 in \cite{nualart book}).
\begin{lemma}
        \label{lemma: prop A.3}
        Let $F\in \mathbb{D}^{1,2}$ and $u\in \mathrm{Dom}(\delta)$ such that
        $Fu\in L^2(\Omega; \mathcal{H}_H)$. Then, $Fu\in \mathrm{Dom}(\delta)$ and it holds
        $$\delta(Fu)=F\delta(u)-\langle DF,u\rangle_{H}.$$

\end{lemma}
\begin{lemma}
        \label{lemma: prop A.4}
Let $u\in L^2(\Omega; \mathcal{H}_H)$ and $\{u_n,\, n\geq 1\}\subset
\mathrm{Dom} (\delta)$ such that
\[
\lim_{n\to\infty}\EE\left[||u_n-u||^2_{\mathcal{H}_H}\right]= 0.
\]
Suppose that there exists a random variable $G\in L^2(\Omega)$ such
that, for all $F\in \mathcal{S}$,
        $$\EE\left[\delta(u_n)F\right]\to \Es{GF}.$$
        Then $u\in \mathrm{Dom}(\delta)$ and $\delta(u)=G$.
\end{lemma}

\medskip

We now define the contraction $\otimes_1$. For $h\in
\mathcal{H}_H^{\otimes n}$ and any element $e_1\otimes \cdots
\otimes e_n$ of the canonical basis of $\mathcal{H}_H^{\otimes n}$,
we define
$$(e_1\otimes \cdots \otimes e_n)\otimes_1 h:=(e_1 \otimes \cdots \otimes e_{n-1}) \langle e_n, h\rangle_H,$$
and we extend it to a generic $f\in \mathcal{H}_H^{\otimes n}$ by
linearity and density. The following lemma can be found in
\cite[Thm. 4.3.8]{paper greeks}:
\begin{lemma}
   \label{lemma: paper greeks}
Let $F\in L^2(\Omega)$ with Wiener chaos expansion $F=\EE[F] +
\sum_{n\geq 1} I_n^H(f_n)$, where $f_n\in \mathcal{H}_H^{\otimes n}$
is symmetric, for all $n\geq 1$. Then $F\in \mathbb{D}^{1,2}$ if and
only if
$$\sum_{n\geq 1} n\, n! \|f_n\|^2_{\mathcal{H}_H^{\otimes n}}<\infty.$$
In this case, for every $h\in \mathcal{H}_H$, we have
$$\langle DF,h\rangle_H=\sum_{n\geq 1} n I_{n-1}^H(f_n\otimes_1 h).$$
\end{lemma}


\subsection{Spectral representation of $W^H$}
\label{sec:spectral}

This section is devoted to prove that any multiple Wiener integral
with respect to the noise $W^H$ admits a representation as a
multiple Wiener integral with respect to a complex-valued Gaussian
measure. For this, we will provide a suitable spectral
representation of $W^H$ in terms of such a complex-valued Gaussian
measure. We point out that the results in the present section will only
be used in Section \ref{subsec: limit identification LMC} in order to identify
the underlying limit law.

Recall that $\{W^H(\varphi),\, \varphi\in\mathcal{H}_H\}$ denotes
the underlying isonormal Gaussian process associated to our noise
$W^H$. Using an approximation argument, one proves that, for any
$t>0$ and $x\in \R$, $1_{[0,t]\times [0,x]}\in \H_H$. Then, we can
define the random field (making an abuse of notation)
\begin{equation}
W^H(t,x):=W^H \left(1_{[0,t]\times [0,x]}\right),\quad (t,x)\in
\R_+\times \R, \label{eq:5}
\end{equation}
which is Gaussian, centered and satisfies, for all $s,t>0$ and
$x,y\in \R$:
\[
\EE\left[W^H(t,x)W^H(s,y)\right]= \frac 12 (s\land
t)\left(|x|^{2H}+|y|^{2H} - |x-y|^{2H}\right).
\]
The latter equality is a consequence of the representation in law of the fractional Brownian
motion as a Wiener type integral with respect to a complex Brownian motion (see, for instance,
\cite[p. 257]{Pipiras 2000}).

Let $\tilde{W}:\mathcal{B}_0(\R_+\times \R) \to \mathbb{C}$ be a
complex-valued Gaussian measure which can be written as
$\tilde{W}=\tilde{W}_1 + i\tilde{W}_2$, where $\tilde{W}_1$ and
$\tilde{W}_2$ are independent real-valued centered Gaussian measures
such that, for any $A,B \in \mathcal{B}(\R_+\times \R)$,
\[
\EE\big[\tilde{W}_j(A) \tilde{W}_j(B) \big] = \frac{|A\cap B|}{2},
\quad j=1,2,
\]
where $|A\cap B|$ is the Lebesgue measure of $A\cap B$. In
particular, $\EE\big[|\tilde{W}(A)|^2\big]=|A|$, for all $A\in
\mathcal{B}_0(\R_+\times \R)$. Note that $\tilde{W}_1$ and
$\tilde{W}_2$ are essentially {\it white noises} in the sense of
\cite[Page 6, Example 3.13]{minicourse}. One can define the integral
of any deterministic function $f\in L^2(\R_+\times \R; \mathbb{C})$
with respect to $\tilde{W}$, as follows:
\[
\int_{\R_+}\int_\R f(t,x) \tilde{W}(dt,dx):= \int_{\R_+}\int_\R
f(t,x) \tilde{W}_1(dt,dx) + i \int_{\R_+}\int_\R f(t,x)
\tilde{W}_2(dt,dx),
\]
and, for $j=1,2$,
\[
\int_{\R_+}\int_\R f(t,x) \tilde{W}_j(dt,dx) := \int_{\R_+}\int_\R
\mbox{Re}[f](t,x) \tilde{W}_j(dt,dx)+ i \int_{\R_+}\int_\R
\mbox{Im}[f](t,x) \tilde{W}_j(dt,dx).
\]
The latter integrals can be interpreted, e.g., as integrals with
respect to a martingale measure (see \cite{Walsh}). It holds that, for any
$f,g\in L^2(\R_+\times \R; \mathbb{C})$,
\[
\EE \left[\int_{\R_+}\int_\R f(t,x) \tilde{W}(dt,dx)
\overline{\int_{\R_+}\int_\R g(t,x) \tilde{W}(dt,dx)}\right] =
\int_{\R_+}\int_\R f(t,x) \overline{g(t,x)} \, dx dt.
\]
This yields, for all $f\in L^2(\R_+\times \R; \mathbb{C})$, the
isometry property
\[
\EE\left[\left|\int_{\R_+}\int_\R f(t,x)
\tilde{W}(dt,dx)\right|^2\right]= \int_{\R_+}\int_\R |f(t,x)|^2 \,
dx dt.
\]

We have the following result, whose proof follows immediately.

\begin{proposition}
\label{prop: representation W^H 2D}
Set, for any $(t,x)\in
\R_+\times \R$,
        \begin{equation}
        \label{eq: spectral representation W^H}
            \tilde{W}^H(t,x):=\sqrt{c_H} \int_{0}^{t}\int_{\R}
            \F[1_{[0,x]}](\xi)|\xi|^{\frac{1}{2}-H}\tilde{W}(ds,d\xi).
        \end{equation}
        Then, $\tilde{W}^H$ is a Gaussian process which has the same distribution
        as the random field $W^H$ defined in \eqref{eq:5}.
\end{proposition}

At this point, we aim to extend the random field $\tilde{W}^H$
defined in \eqref{eq: spectral representation W^H} to an isonormal
Gaussian process in $\H_H$. We need the following corollary of
\cite[Thm. 4.3]{Basse}:

\begin{proposition}
\label{cor: density step functions in H^H}
The space of finite
linear combinations of functions of the form
\[
f(r,z)= 1_{(s,t]\times (x,y]}(r,z),
\]
with $0\leq s<t$ and $x<y$, is dense in the Hilbert space
$\mathcal{H}_H$.
\end{proposition}

\begin{proof}
The result is a direct consequence of \cite[Thm. 4.3]{Basse}.
Indeed, in the latter paper it is proved that any predictable
process $\{X(t,x),\, (t,x)\in \R_+\times \R\}$ belonging to
$L^2(\Omega;\H_H)$ can be approximated by finite linear combinations
of processes of the form $(r,z,\omega)\mapsto 1_G(\omega)
1_{(s,t]}(r)1_{(x,y]}(z)$, for some $G\in \F$. To prove our result,
it suffices to observe that, if we choose a deterministic element
$\varphi$ in their proof, also its approximating sequence
$\varphi_n$ is deterministic, and the norm in the space
$L^2(\Omega;\H_H)$ coincides with the norm in $\H_H$ for
deterministic elements.
\end{proof}

Let us now define, for any $(t,x)\in \R_+\times \R$,
\[
\tilde{W}^H\left(1_{[0,t]\times [0,x]}\right):=\tilde{W}^H(t,x)
\]
(again making an abuse of notation). This definition can be extended
by linearity to any simple function on $\R_+\times \R$. Then, thanks
to Proposition \ref{cor: density step functions in H^H} and using an
approximation argument, one constructs an isonormal Gaussian process
$\{\tilde{W}^H(\varphi),\, \varphi\in \H_H\}$ which has exactly the
same law as $\{W^H(\varphi),\, \varphi\in \H_H\}$.

For the remainder of the paper, we will assume, without any loss of
generality, that our Gaussian setting is the one determined by the
isonormal Gaussian process $\tilde{W}^H=\{\tilde{W}^H(\varphi),\,
\varphi\in \H_H\}$.
For the sake of simplicity, we
will use again the notation $W^H$ instead of $\tilde{W}^H$.
So, the main implications of this setting are
that, first, we have the representation
\begin{equation}
W^H\left(1_{[0,t]\times [0,x]}\right) = \sqrt{c_H}
\int_{0}^{t}\int_{\R}
            \F[1_{[0,x]}](\xi)|\xi|^{\frac{1}{2}-H}\tilde{W}(ds,d\xi)
\label{eq:8}
\end{equation}
and, secondly, the whole family of processes $\{W^H,\, H\in (0,1)\}$
are defined in a single probability space, which is the one where
the Gaussian measure $\tilde{W}$ is defined. This last fact
will be crucial in Section \ref{subsec: limit identification LMC}.

The main result of the section is the following:

\begin{theorem}
\label{teo: representation multiple integral}
Let $n\geq 1$, $f\in \mathcal{H}_H^{\otimes_n}$ and $I^H_n(f)$ be
the multiple Wiener integral of $f$ with respect to  $W^H$. Let
$\hat{f}$ be the function defined by
\[
\hat{f}(t_1,x_1,t_2,x_2,\dots,t_n,x_n) = (c_H)^{\frac{n}{2}}\F
[f(t_1,\cdot,t_2,\cdot,\dots,t_n,\cdot)](x_1,\dots,x_n) \,
|x_1|^{\frac{1}{2}-H}\cdots |x_n|^{\frac{1}{2}-H},
\]
where we recall that the
constant $c_H$ is given in \eqref{eq:103}
Then, it holds that
        \begin{equation}
        \label{eq: representation mult Wiener integral}
        I_n^H(f)=\tilde{I}_n(\hat{f}), \; \mathbb{P}\text{-a.s.},
        \end{equation}
where $\tilde{I}_n$ is the $n$-th order Wiener integral with respect
to the complex Gaussian measure $\tilde{W}$.
\end{theorem}

\begin{proof}
We first check that the result is true for the first-order Wiener
integral $I_1^H$. We aim to prove that, for any $\varphi\in
\mathcal{H}_H$,
\begin{equation}
I_1^H(\varphi)=(c_H)^{\frac{1}{2}} \, \tilde{I}_1\left(\F \varphi(t,\cdot)(x) \,
|x|^{\frac{1}{2}-H}\right),
\label{eq:9}
\end{equation}
 which means that
$$\int_{\R_+}\int_{\R} \varphi(t,x) W^H(dt,dx)=(c_H)^{\frac{1}{2}}
\int_{\R_+}\int_{\R}
\F\varphi(t,\cdot)(x)\,|x|^{\frac{1}{2}-H}\tilde{W}(dt,dx).$$ By
\eqref{eq:8} and the linearity of the Wiener integral, the latter
equality clearly holds in the case where $\varphi(t,x)=
1_{(r,s]\times(y,z]}(t,x)$, for $0\leq r<s$ and $y<z$. Moreover,
owing to Proposition \ref{cor: density step functions in H^H}, it
can also be extended to the whole space $\H_H$, hence proving the
theorem's statement for first-order Wiener integrals.

Let us now prove \eqref{eq: representation mult Wiener integral} for $n>1$. We first
consider the case where $f\in \H_H^{\otimes n}$
is an elementary element of the form \eqref{eq: elementary function}. In this case, we use the definition
of the multiple Wiener integral
(see \eqref{eq: multiple wiener integral}) and the validity of the case $n=1$ (see \eqref{eq:9}), as follows:
\begin{align*}
I^H_n(f)& =   P_{k_1}\left(W^H(e_{j_1})\right)\cdots P_{k_m}\left(W^H(e_{j_m})\right) \\
& =   P_{k_1}\left(\tilde{I}_1(\hat{e}_{j_1})\right) \cdots
P_{k_m}\left(\tilde{I}_1(\hat{e}_{j_m})\right)    \\
& = \tilde{I}_n\left(  \hat{e}_{j_1}^{\otimes_{k_1}} \otimes \cdots
 \otimes \hat{e}_{j_m}^{\otimes_{k_m}} \right) \\
& = \tilde{I}_n (\hat{f}).
\end{align*}
The extension to any element of $\H_H^{\otimes n}$ can be proved by recalling that the set of finite
linear combinations of elementary elements of the form \eqref{eq: elementary function} is dense
in $\H_H^{\otimes n}$.
\end{proof}


\subsection{It\^o and Skorohod stochastic integrals}
\label{subsec: stochastic integral SPDE}

This section is devoted to recall the definition of stochastic
integrals with respect to $W^H$, both in the case $H<\frac12$ and
$H\geq \frac12$, and to prove that the Skorohod integral with
respect to $W^H$ of an adapted process coincides with the
corresponding It\^o integral (see Theorem \ref{teo: equivalence
ito-skorohod} below). This result will allow us to express any
Picard iteration associated to our underlying SDPEs as a finite
sum of multiple Wiener integrals, and this fact will be used in the
proof of Theorem \ref{prop: finite dimensional convergence LMC} in
Section \ref{subsec: limit identification LMC}.

Recall that we have a complete probability space $(\Omega,\F,\P)$ in
which we have our complex-valued Gaussian measure $\tilde{W}$ (see
Section \ref{sec:spectral}). Recall that our isonormal Gaussian
process $W^H=\{W^H(\varphi),\, \varphi\in \H_H\}$ has been defined
in such a way that we may assume that $W^H$ is defined in
$(\Omega,\F,\P)$, for all $H\in (0,1)$. Regarding adaptability, we
consider the natural filtration generated by $\tilde{W}$, which we
denote by $\{\F_t,\,t\geq 0\}$ and can be defined as
$\F_t=\sigma(\tilde{W}(s,x),\, (s,x)\in[0,t]\times \R)$, where
\[
 \tilde{W}(s,x):=\int_{\R_+}\int_\R 1_{[0,s]\times [0,x]}(r,z)\tilde{W}(dr,dz).
\]

\smallskip

Fix a time horizon $T>0$. We denote by $\mathcal{E}$ the space of {\it simple processes} on
$[0,T]\times \R$, that is the space of finite linear combinations of processes of the form
\begin{equation}
            \label{eq: elementary process}
            g(t,x,\omega):=Y(\omega)1_{(r,s]\times(y,z]}(t,x),
\end{equation}
        for some $0\leq r\leq s\leq T$ and $y\leq z$, and for some $\F_r$-measurable random variable $Y$.
The (It\^o) stochastic integral of $g$ with respect to $W^H$ is defined as follows: for any $t\in [0,T]$, set
\begin{equation*}
 \int_{0}^{t}\int_{\R} g(\tau,x) W^H(d\tau,dx)
        := Y\left( W^H(t\wedge s,z)-W^H(t\wedge s,y)-W^H(t\wedge r,z)+W^H(t\wedge r,y) \right).
\end{equation*}
This definition can be extended to all elements of $\mathcal{E}$ by
linearity. Following \cite{Dalang} and \cite{BJQ15}, we endow
$\mathcal{E}$ with the norm
\begin{equation*}
 \|g\|_0:=\left( \EE\left[c_H \int_{0}^{T}\int_{\R} |\F g(t,\cdot)(\xi)|^2|\xi|^{1-2H}d\xi
 dt\right]\right)
  ^\frac12,
\end{equation*}
and we define $\mathcal{P}_0^T$ as the completion of $\mathcal{E}$ with respect to the norm $\|\cdot\|_0$.
It turns out that $\mathcal{P}_0^T$ is the space of predictable processes $g$ for which
$\|g\|_0<\infty$. The stochastic integral can be extended to the whole space $\mathcal{P}_0^T$.

The following result is a particular case of \cite[Prop.
2.9]{Quer-Dalang}:

    \begin{theorem}
        \label{teo: Integrability condition H>1/2}
Suppose that $H\in [\frac{1}{2},1)$. Let $\Gamma:[0,T]\times \R\to
\R$ be such that, for all $t\in(0,T]$, the function
$\Gamma(t,\cdot)$ defines non-negative distribution with rapid
decrease and
$$\int_{0}^{T}\int_{\R} |\F \Gamma(t,\cdot)(\xi)|^2 |\xi|^{1-2H}d\xi dt <\infty.$$
Moreover, we assume that, for all $t\in [0,T]$, $\Gamma(t,dx):=\Gamma(t,x)dx$ defines a
non-negative measure on $\R$ such that
$$\sup_{t\in[0,T]} \Gamma(t, \R)<\infty.$$
Let $Z=\{Z(t,x),\, (t,x)\in [0,T]\times \R\}$ be a predictable
stochastic process satisfying
$$\sup_{(t,x)\in[0,T]\times \R} \EE\left[|Z(t,x)|^2\right] <\infty.$$
Then, the process $\{S(t,x):=Z(t,x)\Gamma(t,x), \, (t,x)\in
[0,T]\times \R \}$ belongs to $\mathcal{P}_0^T$. Furthermore, if $Z$
satisfies, for some $p\geq 2$, that
        $$\sup_{(t,x)\in[0,T]\times \R} \EE\left[|Z(t,x)|^p\right]<\infty,$$
then we have the following Burkholder-Davis-Gundy's inequality:
\begin{align}
\label{eq: Burkholder H>1/2}
& \EE \left[\left|\int_0^T\int_\R S(s,x) W^H(ds,dx)\right|^p\right] \\
& \qquad \quad \leq z_p (\nu_{T,H})^{\frac{p}{2}-1} \int_{0}^{T}
\sup_{x \in \R} \EE\left[|Z(s,x)|^p\right]  \int_{\R} c_H\,
|\F\Gamma(s,\cdot)(\xi)|^2|\xi|^{1-2H}d\xi ds,  \nonumber
\end{align}
where the constant $z_p$ is the one in the classical
Burkholder-Davis-Gundy inequality for continuous martingales, and
$\nu_{T,H}$ is given by
$$\nu_{T,H}:=c_H\int_{0}^{T} \int_{\R} |\F \Gamma(s,\cdot)(\xi)|^2 |\xi|^{1-2H}d\xi ds.$$
\end{theorem}

As far as the case $H<\frac12$ is concerned, we have the following
result (see \cite[Thm. 2.9]{BJQ15}).

\begin{theorem}
        \label{teo: integrability condition H<1/2}
Suppose that $H\in(0,\frac{1}{2})$. Let $\{S(t,x),\, (t,x)\in
[0,T]\times \R\}$ be a predictable process such that, for every
$(\omega,t)$, $S(\omega,t,\cdot)$ defines a tempered function whose
Fourier transform $\F S(\omega,t, \cdot)$ is a locally integrable
function satisfying
$$\EE\left[\int_{0}^{T}\int_{\R} |\F S(t,\cdot)(\xi)|^2 |\xi|^{1-2H}d\xi dt\right]<\infty.$$
Then, $S\in \mathcal{P}_0^T$ and we have the isometry
\[
\EE\left[\left| \int_0^T\int_\R S(t,x) W^H(dt,dx)\right|^2\right] =
\EE\left[\int_{0}^{T}\int_{\R} |\F S(t,\cdot)(\xi)|^2
c_H\,|\xi|^{1-2H}d\xi dt\right].
\]
Moreover, we have the Burkholder-Davis-Gundy inequality: for any
$p\geq 2$,
\begin{equation}
\label{eq: Burkholder H<1/2} \EE\left[\left| \int_0^T\int_\R S(t,x)
W^H(dt,dx)\right|^p\right] \leq z_p \, c_H^{\frac{p}{2}} \,
\EE\left[ \int_{0}^{T}\int_{\R} |\F S(t,\cdot)(\xi)|^2 |\xi|^{1-2H}
d\xi dt \right]^{\frac{p}{2}},
\end{equation}
where the constant $z_p$ is the constant appearing in the classical
Burkholder-Davis-Gundy inequality for continuous martingales.
\end{theorem}

\begin{remark}\label{rmk:99}
Owing to \cite[Prop. 2.8]{BJQ15}, the isometry property in the above
Theorem \ref{teo: integrability condition H<1/2} can be equivalently
written as
$$\EE\left[\left| \int_0^T\int_\R S(t,x) W^H(dy,dx)\right|^2\right] =
\EE\left[\tilde{c}_H \int_{0}^{T}\int_{\R^2} |S(t,x)-S(t,y)|^2 |x
-y|^{2H-2}dydx dt\right],$$ where $\tilde{c}_H=\frac{H(1-2H)}{2}$.
Hence, \eqref{eq: Burkholder H<1/2} becomes
$$\EE\left[\left| \int_0^T\int_\R S(t,x) W^H(dy,dx)\right|^p\right]\leq
 z_p \, \tilde{c}_H^{\frac{p}{2}}\,
\EE\left[ \int_{0}^{T}\int_{\R} |S(t,x)-S(t,y)|^2 |x-y|^{2H-2}
dydxdt\right]^{\frac{p}{2}}.$$
\end{remark}

\medskip

The following result is an extension of \cite[Thm.
4.2]{intermittency} to the case $H> \frac 12$. Note that, in this
latter case, though the noise is more regular in space than a white
noise, the corresponding Hilbert space $\H_H$ may be rather big, and
indeed contains genuine distributions. This makes our proof
different compared to the one of \cite[Thm. 4.2]{intermittency}, in
which $\H_H$ is a space of functions (because $H<\frac12$).

\begin{theorem}
\label{teo: equivalence ito-skorohod} Let $H\in[\frac{1}{4},1)$ and
$u=\{u(t,x), \, (t,x)\in[0,\infty)\times \R\}$ be a stochastic
process such that, restricted to $t\in[0,T]$, belongs to
$\mathcal{P}_0^T$. Then, for any $t>0$, $u1_{[0,t]}\in
\mathrm{Dom}(\delta)$ and its Skorohod integral coincides with the
It\^o integral, that is
 \begin{equation*}
        \int_{0}^{\infty} \int_{\R} u(s,x)1_{[0,t]}(s)W^H(\delta s, \delta x)=
        \int_{0}^{t}\int_{\R} u(s,x)W^H(ds,dx), \quad
        \mathbb{P}\text{-a.s.}
        \end{equation*}
\end{theorem}

\begin{proof}
The proof is an adaptation of that of \cite[Thm.
4.2]{intermittency}. The only difference is that, here, a general
element of $\mathcal{H}_H$ is not necessarily a function. It is enough to prove
the statement in the case where $u$ is an elementary process of the
form \eqref{eq: elementary process}. The extension to any arbitrary
element of $\mathcal{P}_0^T$ can be done exactly as in Case 2 of the
proof of \cite[Thm. 4.2]{intermittency}.

Let $g$ be an elementary process of the form
$g(\tau,x,\omega)=Y(\omega)1_{(r,s]}(\tau)1_{(y,z]}(x)$, with $0\leq
r< s\leq T$ and $y<z$, where we assume that $Y$ is
$\F_r$-measurable, bounded and belongs to $\mathbb{D}^{1,2}$. We have to check
that $g1_{[0,t]} \in \text{Dom}(\delta)$ and it holds
\[
\delta\left(g1_{[0,t]}\right)= \int_{0}^{t}\int_{\R}
g(\tau,x)W^H(d\tau,dx).
\]
First, we note that $g1_{[0,t]}=Y 1_{[r\wedge t ,s\wedge
t]\times[y,z]}$. Since $Y\in \mathbb{D}^{1,2}$ and $1_{[r\wedge t
,s\wedge t]\times[y,z]}\in \text{Dom}(\delta)$, we can apply Lemma
\ref{lemma: prop A.3} to conclude that $g1_{[0,t]}\in
\text{Dom}(\delta)$ and
$$\delta(g1_{[0,t]})=Y \delta(1_{[r\wedge t ,s\wedge t]\times[y,z]})-
\langle DY, 1_{[r\wedge t ,s\wedge t]\times[y,z]} \rangle_H,$$
 if
the right-hand side above belongs to $L^2(\Omega)$. We clearly have
that $Y \delta(1_{[r\wedge t ,s\wedge t]\times[y,z]})\in
L^2(\Omega)$, and we will show that $\langle DY, 1_{[r\wedge t
,s\wedge t]\times[y,z]} \rangle_H=0$, which will allow us to
conclude the proof.

Let $h:=1_{[r\wedge t ,s\wedge t]\times[y,z]}$. Since $Y$ is
$\F_r$-measurable, we have, by Lemma \ref{lemma:
A.1},
$$Y=\Es{Y|\F_r}=\sum_{n\geq 0}I^H_n\big(g_n1_{[0,r]}^{\otimes  n}\big),$$
for some symmetric $g_n\in \mathcal{H}_H^{\otimes n}$, $n\geq 1$. By
Lemma \ref{lemma: paper greeks} we have that
$$\langle DY,h\rangle_H=\sum_{n\geq 1} n I^H_{n-1}\big(g_n1_{[0,r]}^{\otimes n}\otimes_1 h\big).$$
We claim that $g1_{[0,r]}^{\otimes n}\otimes_1 h=0$, for all $g\in
\mathcal{H}_H^{\otimes n}$. Indeed, if $g=e^{\otimes n}$ for some
function $e\in \mathcal{H}_H$, we have
$$e^{\otimes n} 1_{[0,r]}^{\otimes n}\otimes_1 h=e^{\otimes (n-1)}  1_{[0,r]}^{\otimes (n-1)}
\langle e 1_{[0,r]}, h\rangle_H,$$ and we observe that
$$\langle e 1_{[0,r]}, h\rangle_H=\int_{0}^{\infty}\int_{\R} \F e(s,\cdot)(\xi) 1_{[0,r]}(s)
\overline{\F 1_{[y,z]}(\xi)}1_{[r\wedge t, s\wedge t]}(s)d\xi
ds=0.$$ This can be extended to a generic element in
$\mathcal{H}_H^{\otimes n}$ by linearity and density (using Lemma \ref{lemma: prop A.4}).
\end{proof}


\subsection{Existence and uniqueness of solution}

This section is devoted to recall the well-posedness results for
equations \eqref{wave} and \eqref{heat} and prove that the
corresponding Picard iterations admit a suitable  finite Wiener
chaos expansion.

First, we recall that the solution to our equations is understood in
the mild sense. Namely, an adapted and jointly measurable random
field $u^H=\{u^H(t,x),\, (t,x)\in[0,T]\times \R\}$ solves
\eqref{wave} (resp. \eqref{heat}) if it holds, for all $(t,x)\in
[0,T]\times \R$:
\begin{equation}
\label{eq: mild formulation LMC} u^H(t,x)= \eta+\int_{0}^{t} \int_\R
G_{t-s}(x-y) u^H(s,y)W^H(ds,dy),
\end{equation}
where $G$ is the fundamental solution of the wave (resp. heat)
equation in $\R$ (see \eqref{eq:3}).

\smallskip

The following result is a particular case of \cite[Thm.
4.3]{Quer-Dalang}, which covers the case $H\geq \frac12$.

\begin{theorem}
    \label{teo: existence-uniqueness LMC H>1/2}
    Let $H\in[\frac{1}{2},1)$. There exists a unique mild solution $u^H$ to equation
    \eqref{eq: mild formulation LMC}. Moreover, the solution $u^H$ is $L^2(\Omega)$-continuous
    and satisfies, for every $p\geq 1$,
    $$\sup_{(t,x)\in [0,T]\times \R} \EE\left[|u^H(t,x)|^p\right]<\infty.$$
\end{theorem}

\begin{remark}
The case $H=\frac12$ corresponds to the space-time white noise,
while in the case $H\in(\frac{1}{2},1)$ the noise's spatial
correlation is given by a Riesz kernel of order $2-2H$.
\end{remark}

The case $H\in (\frac14,\frac12)$ has been considered in \cite[Thm.
1.1]{BJQ15}. In the latter reference, the authors proved that
condition $H>\frac14$ is necessary and sufficient in order to have a
solution (see \cite[Prop. 3.7]{BJQ15}).

\begin{theorem}
    \label{teo: existence-uniqueness LMC H<1/2}
    Let $H\in(\frac{1}{4},\frac{1}{2})$. There exists a unique mild solution $u^H$ to
    \eqref{eq: mild formulation LMC}. Moreover, the solution $u^H$ is $L^2(\Omega)$-continuous
    and satisfies, for every $p\geq 2$,
    \begin{equation}
    \label{eq: L^p boundedness solution}
    \sup_{(t,x)\in [0,T]\times \R}\EE\left[|u^H(t,x)|^p\right]<\infty
    \end{equation}
    and
    \begin{equation}
    \label{eq: Sobolev norm Thm 1.1}
    \sup_{(t,x)\in [0,T]\times \R} \int_{0}^{T}\int_{\R^2} G^2_{t-s}(x-y)
    \frac{\EE\left[|u^H(s,y)-u^H(s,z)|^p\right]^{\frac{2}{p }}}{|y-z|^{2-2H}}dydzds<\infty.
    \end{equation}

\end{theorem}

\medskip

\begin{remark}\label{rmk:1}
In the case $H\in(\frac{1}{4},\frac{1}{2})$, the solution $u^H$
satisfies, in addition to \eqref{eq: L^p boundedness solution}, the
further constraint \eqref{eq: Sobolev norm Thm 1.1}. This comes from
the fact that, in \cite{BJQ15}, the solution of \eqref{eq: mild
formulation LMC} was proved to exist in the space of
$L^2(\Omega)$-continuous, adapted and jointly measurable processes
endowed with a Sobolev's type norm which included a term of the form
\eqref{eq: Sobolev norm Thm 1.1}.
\end{remark}

In the case $H\in (\frac 14, \frac 12)$, the solution $u^H$ of
\eqref{wave} (and \eqref{heat}) has been found in \cite{BJQ15} as a
limit of the Picard iteration scheme, which is defined by
\begin{equation*}
\begin{split}
u_0^H(t,x) & := \eta \\
u_{m+1}^H(t,x) & :=  \eta +\int_{0}^{t}\int_{\R}
G_{t-s}(x-y)u^H_m(s,y)W^H(ds,dy), \quad m \geq 0,
\end{split}
\end{equation*}
where $(t,x)\in [0,T]\times \R$. The limit is found in the Banach
space $\chi^p_H$, for $p\geq 2$, which is defined as the space of
$L^2(\Omega)$-continuous, adapted and jointly measurable processes
$Y=\{Y(t,x),\,(t,x)\in[0,T]\times\R\}$ such that
$$\|Y\|_{\chi^p_H}:=\|Y\|_{\chi^p_1}+\|Y\|_{\chi^p_{H,2}}<\infty,$$
where,
$$\|Y\|_{\chi^p_1}:=\sup_{(t,x)\in[0,T]\times \R}
\EE\left[|Y(t,x)|^p\right]^{\frac{1}{p}}$$ and
\begin{equation*}
\|Y\|_{\chi^p_{H,2}} := \sup_{(t,x)\in[0,T]\times \R} \Bigg(
\tilde{c}_H \int_{0}^{t} \int_{\R^2} G_{t-s}^2(x-y) \frac{\left(
\EE\left[|Y(s,y)-Y(s,z)|^p\right]
\right)^{\frac{2}{p}}}{|y-z|^{2-2H}} \,dydzds\Bigg)^{\frac{1}{2}}.
\end{equation*}
We recall that the constant $\tilde{c}_H$ has been defined in Remark
\ref{rmk:99}. Notice that the $L^p$-part $\|\cdot\|_{\chi^p_1}$ of
the norm $\|\cdot\|_{\chi^p_H}$ does not depend on $H$, as it is
also pointed out by the notation itself, while the
\textit{Gagliardo-type} part $\|\cdot\|_{\chi^p_{H,2}}$ does depend
on $H$.

\begin{remark}
    \label{rem: normalizing constant C_H for the norm}
In \cite{BJQ15}, the norm $\|\cdot\|_{\chi^p_{H,2}}$ is defined
without the constant $\tilde{c}_H=\frac{H(1-2H)}{2}$. Since the two
definitions give rise to equivalent norms, the results about
existence and uniqueness of solution for equation \eqref{eq: mild
formulation LMC} when $H\in(\frac14,\frac12)$ still hold true. On
the other hand, we will see how adding this normalizing constant
helps us proving some uniform (in $H$) results that will be needed
in the sequel.
\end{remark}

Before stating the main result of the section, we consider the
following Sobolev-type embedding for the space $\chi^p_H$, which
could be of independent interest.
\begin{lemma}
    \label{lemma: sobolev embedding}
Let $p\geq 2$ and $\frac{1}{4}<\alpha\leq \beta<\frac{1}{2}$. Then,
it holds:
 $$\chi^p_\alpha \hookrightarrow \chi^p_\beta.$$
This means that there exists a constant $C$ such that, for every
adapted, jointly measurable and $L^2(\Omega)$-continuous process
$Y$, we have
    \begin{equation}
    \label{eq: inequality sobolev embedding}
    \|Y\|_{\chi^p_\beta}\leq C \|Y\|_{\chi^p_\alpha}.
    \end{equation}
Moreover, it holds the following stronger property for the
Gagliardo-type seminorm $\|\cdot\|_{\chi^p_{\beta,2}}$:
    \begin{equation*}
    \sup_{\beta\in[\alpha,\frac{1}{2})} \|Y\|_{\chi^p_{\beta,2}}\leq \tilde{C}\|Y\|_{\chi^p_\alpha}
    \end{equation*}
    where the constant $\tilde{C}$ only depends on $p$ and $T$.
\end{lemma}
\begin{proof}
We follow the same lines as in the proof of
\cite[Prop. 2.1]{hitch}. It suffices to prove
\eqref{eq: inequality sobolev embedding} for the
$\|\cdot\|_{\chi^p_{H,2}}$-part of the norm. It holds:
\begin{align}
 \label{eq: first estimate sobolev embedding}
 & \left( \tilde{c}_\beta \int_{0}^{t}\int_{\R^2}  G^2_{t-s}(x-y)
 \frac{\left(\EE\left[|Y(s,y)-Y(s,z)|^p\right]\right)^{\frac{2}{p}}}{|y-z|^{2-2\beta}}\, dydzds\right)^{\frac{1}{2}} \nonumber \\
 &\qquad = \left( \tilde{c}_\beta \int_{0}^{t}\int_{\R^2}  G^2_{t-s}(x-y)
     \frac{\left(\EE\left[|Y(s,y)-Y(s,y-\overline{z})|^p\right]\right)^{\frac{2}{p}}}{|\overline{z}|^{2-2\beta}}\,
     dyd\overline{z}ds\right)^{\frac{1}{2}}\nonumber \\
&\qquad \leq  C(I_1+I_2),
\end{align}
    where we label $I_1$ the term where we integrate in the variable $\overline{z}$
    in the region $|\overline{z}|\geq1$, and $I_2$ the term where we integrate in the region
    $|\overline{z}|<1$. First, we have
    \begin{equation*}
    \begin{split}
    I_1= &  \left( \tilde{c}_\beta  \int_{0}^{t}\int_{\R} \int_{|\overline{z}|\geq 1}
    G^2_{t-s}(x-y)\frac{\left(\EE\left[|Y(s,y)-Y(s,y-\overline{z})|^p\right]\right)^{\frac{2}{p}}}{|\overline{z}|^{2-2\beta}}
    \, d\overline{z}dyds \right)^{\frac{1}{2}} \\
    & \quad \leq  C_p \sup_{(t,x)\in[0,T]\times \R}  \left(\EE\left[|Y(t,x)|^p\right]\right)^{\frac{1}{p}}
    \Bigg(\tilde{c}_\beta  \int_{0}^{t}\int_{\R} \int_{|\overline{z}|\geq 1}
    G^2_{t-s}(x-y)\frac{1}{|\overline{z}|^{2-2\beta}}\, d\overline{z}dyds\Bigg)^{\frac{1}{2}} \\
    \end{split}
    \end{equation*}
    Note that  $\int_{|\overline{z}|\geq 1}\frac{1}{|\overline{z}|^{2-2\beta}}d\overline{z}=
    \frac{2}{1-2\beta}$. Hence,
    \begin{equation*}
    \begin{split}
    &  \tilde{c}_\beta \int_{0}^{t}\int_{\R} \int_{|\overline{z}|\geq 1}
    G^2_{t-s}(x-y)\frac{1}{|\overline{z}|^{2-2\beta}}\, d\overline{z}dyds \\
     & \qquad \leq  \beta\int_{0}^{t}\int_{\R} G^2_{t-s}(x-y)\, dyds \leq \beta C_T \leq \frac{C_T}{2}.
    \end{split}
    \end{equation*}
    Thus, we can conclude that
    $$I_1\leq C_{p,T} \sup_{(t,x)\in[0,T]\times \R} \left(\EE\left[|Y(t,x)|^p\right]\right)
    ^{\frac{1}{p}}.$$
    Regarding $I_2$, we observe that
    \begin{align*}
    I_2 & =  \left( \tilde{c}_\beta  \int_{0}^{t}  \int_{\R} \int_{|\overline{z}|< 1}
    G^2_{t-s}(x-y)  \frac{\left(\EE\left[|Y(s,y)-Y(s,y-\overline{z})|^p\right]\right)^{\frac{2}{p}}}
    {|\overline{z}|^{2-2\beta}}\, d\overline{z}dyds \right)^{\frac{1}{2}} \\
     & \leq \left( \tilde{c}_\alpha  \int_{0}^{t}  \int_{\R} \int_{|\overline{z}|< 1}
     G^2_{t-s}(x-y)  \frac{\left(\EE\left[|Y(s,y)-Y(s,y-\overline{z})|^p\right]\right)^{\frac{2}{p}}}
     {|\overline{z}|^{2-2\alpha}}\, d\overline{z}dyds \right)^{\frac{1}{2}} \\
    & \leq \left( \tilde{c}_\alpha  \int_{0}^{t}  \int_{\R} \int_{\R}   G^2_{t-s}(x-y)
    \frac{\left(\EE\left[|Y(s,y)-Y(s,y-\overline{z})|^p\right]\right)^{\frac{2}{p}}}{|\overline{z}|^{2-2\alpha}}
    \, d\overline{z}dyds \right)^{\frac{1}{2}} \\
     & \leq \sup_{(t,x)\in[0,T]\times \R} \left( \tilde{c}_\alpha   \int_{0}^{t}  \int_{\R}
     \int_{\R}   G^2_{t-s}(x-y)  \frac{\left(\EE\left[|Y(s,y)-Y(s,y-\overline{z})|^p\right]\right)^{\frac{2}{p}}}
     {|\overline{z}|^{2-2\alpha}}\, d\overline{z}dyds \right)^{\frac{1}{2}} \\
    & = \|Y\|_{\chi^p_{\alpha,2}}.
    \end{align*}
    Notice that both the estimate for $I_1$ and $I_2$ are independent of $(t,x)\in[0,T]\times \R$ and
    $\beta\in[\alpha,\frac{1}{2})$. Therefore, we can take the supremum with respect to
    $(t,x)\in [0,T]\times \R$ and $\beta\in[\alpha,\frac{1}{2})$ in the left-hand side of
    \eqref{eq: first estimate sobolev embedding} and we conclude
    \begin{equation*}
    \sup_{\beta\in[\alpha,\frac{1}{2})} \|Y\|_{\chi^p_{\beta,2}}\leq C_{p,T}\|Y\|_{\chi^p_1}+
    \|Y\|_{\chi^p_{\alpha,2}}\leq \tilde{C}\|Y\|_{\chi^p_\alpha},
    \end{equation*}
    which obviously implies
    \begin{equation*}
    \|Y\|_{\chi^p_\beta}\leq (C_{p,T}+1)\|Y\|_{\chi^p_1}+\|Y\|_{\chi^p_{\alpha,2}}
    \leq C \|Y\|_{\chi^p_\alpha},
    \end{equation*}
    for some constant $C$.
\end{proof}

The path H\"older-continuity of the solution to \eqref{eq: mild
formulation LMC} has been proved in \cite{holder} in the case $H\in
(\frac{1}{4},\frac{1}{2})$, while the case $H\in[\frac{1}{2},1)$
follows from the results in \cite{Walsh,SaSa00,SaSa02}. For the sake
of completeness, we state a result which unifies both cases, and
whose proof follows, indeed, as an immediate consequence of the
stronger results Proposition \ref{prop: tightness LMC H<1/2} and
Proposition \ref{prop: tightness LMC H>1/2} proven in Section
\ref{sec:tightness}.

\begin{theorem}
    Let $H\in (\frac{1}{4},1)$.
    Then, the solution $u^H$ to \eqref{eq: mild formulation LMC} satisfies
    the following: for any $p\geq 2$, there exists a constant $C_p>0$
    (which indeed does not depend on $H$) such that, for all
    $t,t'\in [0,T]$ and $x,x'\in \R$, it holds
    \begin{equation*}
    \sup_{x\in\R}\EE\left[|u^H(t',x)-u^H(t,x)|^p\right]\leq C_p|t'-t|^{\gamma p}
    \end{equation*}
    and
    \begin{equation*}
    \sup_{t\in[0,T]}\EE\left[|u^H(t,x')-u^H(t,x)|^p\right]\leq C_p|x'-x|^{Hp},
    \end{equation*}
     where $\gamma=H$ for the wave equation and $\gamma=\frac H2$ for the heat equation. Thus, the process $u^H$ has a modification
     whose trajectories are almost surely $\gamma'$-H\"older continuous in
     time, for all $\gamma'<\gamma$, and $H'$-H\"older continuous in
     space for all $H'<H$.
\end{theorem}

\begin{proof}
    As already mentioned, the result follows from
    Propositions \ref{prop: tightness LMC H<1/2} and \ref{prop: tightness LMC H>1/2}
    in Section \ref{sec:tightness}, where the same kind of estimates
    have obtained uniformly with respect to $H$.
\end{proof}

The above Theorems \ref{teo: existence-uniqueness LMC H<1/2} and
\ref{teo: existence-uniqueness LMC H>1/2}, together with Theorem
\ref{teo: equivalence ito-skorohod} on the equivalence between It\^o
and Skorohod integrals, allow us to prove that equations
\eqref{wave} and \eqref{heat} admit a unique \textit{Skorohod mild
solution}. By definition, it is a square integrable random field
$\{u^H(t,x),\, (t,x)\in \R_+\times \R\}$ such that, for all
$(t,x)\in \R_+\times \R$,
\begin{equation}
\label{eq: mild formulation LMC - skorohod} u^H(t,x)=
\eta+\int_{0}^t \int_\R  G_{t-s}(x-y)u^H(s,y)W^H(\delta s,\delta y),
\quad \mathbb{P}\text{-a.s.},
\end{equation}
that is, the process
$v^{(t,x)}:=\{1_{[0,t]}(s)G_{t-s}(x-y)u^H(s,y),\, (s,y)\in
\R_+\times \R\}$ belongs to $\text{Dom}(\delta)$ and
$u^H(t,x)=\eta+\delta\left(v^{(t,x)}\right)$.

\begin{theorem}
    Let $H\in(\frac 14,1)$ and $T>0$.
    Equation \eqref{eq: mild formulation LMC - skorohod} admits a unique adapted solution
    in $[0,T]\times \R$.
\end{theorem}
\begin{proof}
This result has already been proved in \cite[Thm.
4.3]{intermittency} for the wave equation in the case $H\in(\frac
14,\frac 12)$. In \cite[p. 49]{nualart H<1/2}, the authors notice
that it is also true for the heat equation, still under the
constraint $H\in (\frac 14 , \frac 12)$. The statement's validity in
the case $H\in [\frac 12 ,1)$ follows combining Theorems \ref{teo:
equivalence ito-skorohod} and \ref{teo: existence-uniqueness LMC
H>1/2}.
\end{proof}

Finally, the following result will be crucial in order to identify
the limit law in Theorem \ref{prop: finite dimensional convergence
LMC}.

\begin{theorem}
    \label{teo: equivalence ito-skorohod solution}
Let $H\in (\frac 14,1)$ and $u^H$ be the solution to \eqref{eq: mild
formulation LMC}. Recall that the corresponding Picard iteration
scheme is defined as follows: for any $m\geq 0$, set
    \begin{equation*}
        \begin{split}
            u_0^H(t,x) & := \eta, \\
            u_{m+1}^H(t,x) & := \eta+\int_{0}^{t}\int_{\R}
            G_{t-s}(x-y)u^H_m(s,y)W^H(ds,dy),
        \end{split}
    \end{equation*}
where $(t,x)\in [0,T]\times \R$. Then, for any $m\geq 0$, it holds
$$u_m^H(t,x)=\sum_{n=0}^{m} I^H_n(g_n(\boldsymbol{\cdot},t,x)),$$
where $I^H_n$ is the $n$-th multiple Wiener integral with respect to
$W^H$ and the kernel $g_n(\cdot,t,x)$ is given by
\begin{equation}
\label{eq: g_n} g_n(t_1,x_1,t_2,x_2,\dots,t_n,x_n,t,x):=
G_{t-t_n}(x-x_n)\cdots G_{t_2-t_1}(x_2-x_1) \, \eta
1_{\{0<t_1<\cdots<t_n<t\}}.
\end{equation}
\end{theorem}

\begin{proof}
The case of the wave equation with $H<\frac 12$ has already been
proved in \cite[Thm. 4.3]{intermittency}. Owing to Theorem \ref{teo:
equivalence ito-skorohod}, the arguments in the proof of the former
theorem can be carried out to easily extend the result to the case
$H\geq \frac 12$ as well as to the heat equation.
\end{proof}


\section{Tightness}
\label{sec:tightness}

Recall that our main result (see Theorem \ref{thm:main}) states
that, if $H_0\in (\frac{1}{4},1)$ and $\{H_n,\,n\in\N\}\subset
(\frac 14,1)$ is any sequence converging to $H_0$, then  $u^{H_n}\to
u^{H_0}$ in law in the space $\C([0,T]\times \R)$ of continuous
functions. The first step in order to prove the above result
consists in checking that the laws of $\{ u^{H_n},\, n\in\N \}$
define a tight family of probability measures on
$\C([0,T]\times\R)$.

We split the computations in the case $H_0\in (\frac 14, \frac 12]$,
which has more involved calculations, and the case $H_0\in [\frac
12,1)$, in which the calculations are more straightforward. We
explain briefly why: in the {\it rough} case, the
Burkholder-Davis-Gundy inequality \eqref{eq: Burkholder H<1/2}
forces us to consider the Fourier transform of the whole integrand
process, while in the case $H\in [\frac 12,1)$, when we use the
Burkholder-Davis-Gundy inequality \eqref{eq: Burkholder H>1/2}, we
only have to compute the Fourier transform of the deterministic part
of the integrand process, which will be explicit in our case.


\subsection{Tightness in the case $(\frac 14, \frac 12)$}
\label{subsec: tightness H<1/2 LMC}

We suppose that the limiting Hurst exponent
$H_0\in(\frac{1}{4},\frac{1}{2}]$. If
$H_0\in(\frac{1}{4},\frac{1}{2})$, we can assume without loss of
generality that the whole sequence $\{H_n, \, n\in\N\}\subset
[\eta_1,\eta_2]\subset (\frac{1}{4},\frac{1}{2})$. If
$H_0=\frac{1}{2}$, we can assume at most that $\{ H_n,\, n\in\N
\}\subset [\eta_1,\frac{1}{2})\subset (\frac{1}{4},\frac{1}{2})$.
From now on we will denote both type of sets as $K$, meaning that
$K=[\eta_1,\eta_2]$ if $H_0\in(\frac{1}{4},\frac{1}{2})$ and
$K=[\eta_1,\frac{1}{2})$ if $H_0=\frac{1}{2}$. Clearly, if the
limiting exponent $H_0=\frac 12$, we cannot suppose that $H_n\to
H_0$ always from below. In Section \ref{subsec: tightness H>1/2
LMC}, we will also handle families of Hurst exponents with $K=(\frac
12, \eta_2]$, so that our result will be complete (because the union
of a finite number of tight families is a tight family itself).

We are ready to state the main result of the present section.

\begin{proposition}
    \label{prop: tightness LMC H<1/2}
    Let $\mathcal{U}_K:=\{u^{H},\, H\in K\}$ be the family of solutions of
    \eqref{eq: mild formulation LMC}, where $K$ is  either of the form $[\eta_1,\eta_2]$,
    with $\eta_1,\eta_2\in (\frac 14, \frac 12)$ and $\eta_1 < \eta_2$, or $K=[\eta_1,\frac 12)$, where $\eta_1\in (\frac 14, \frac 12)$. Then, the family $\mathcal{U}_K$ is tight in $\C([0,T]\times \R)$, endowed with the metric of uniform convergence on compact sets.
\end{proposition}

We postpone the proof of this result, since we need some preliminar
results. We aim to apply the tightness criterion Theorem
\ref{th: centsov}. Indeed, we will check that conditions (i) and (ii) in the latter result
are satisfied by the Picard iterations $u_m^H$, uniformly with respect to $H$, and then we will pass to the limit
as $m\to\infty$.

First of all, we show that the the Picard iterations $\{u^H_m,\,
m\geq 0\}$ are well-defined and satisfy some estimates uniformly
with respect to $H$. The proof is very similar to that of \cite[Thm.
3.7]{BJQ15}. In fact, we will follow the same steps in its proof and
take care of the fact that we need all estimates uniformly in $H$.
Only the most significant parts of the proof will be written
explicitly.

\begin{proposition}
    \label{prop: uniform estimates for picard iterations}
    Let $p\geq 2$ and $H\in (\frac14,\frac12)$. For any $m\geq 0$, we have that
\begin{itemize}
\item[(i)] $u_m^H(t,x)$ is well-defined, for any $H\in K$ and  $(t,x)\in[0,T]\times \R$.
\item[(ii)] It holds
\[
\sup_{H\in K} \sup_{(t,x)\in[0,T]\times \R} \EE\left[|u_m^H(t,x)|^p\right]<\infty.
\]
\item[(iii)] It holds
\[
 \sup_{H\in K} \sup_{(t,x)\in[0,T]\times \R}    \tilde{c}_H \int_{0}^t \int_{\R^2} G^2_{t-s}(x-y)
    \frac{\left(\EE\left[|u_m^H(s,y)-u_m^H(s,z)|^p\right]\right)^{\frac{2}{p}}}{|y-z|^{2-2H}}\, dydzds<\infty.
\]
\end{itemize}
\end{proposition}

\begin{proof}
Condition (i) is a direct consequence of \cite[Thm. 3.7]{BJQ15}. In order to prove (ii) and (iii),
we use an induction argument. First, note that these two conditions clearly hold for $m=0$.

Assume that conditions (ii) and (iii) are satisfied by $u_m^H$. We prove that they are also
fulfilled by $u^H_{m+1}$. Precisely, arguing as in Step 2 in the proof of \cite[Thm. 3.7]{BJQ15}
(see p. 18 therein), we have
\begin{align*}
\EE\left[|u_{m+1}^H(t,x)|^p\right] & \leq C\left\{ \eta^p +
\EE\left[  \left| \tilde{c}_H \int_{0}^{T} \int_{\R^2}
\frac{|S_m^H(s,y)-S_m^H(s,z)|^2}{|y-z|^{2H-2}}\, dydzds
\right|^{\frac{p}{2}} \right]\right\},
\end{align*}
where we have used the notation $S_m^H(s,y):=G_{t-s}(x-y)u^H_m(s,y)$ and $C$ is some positive constant.
The expectation on the right hand-side above can be bounded, up to some constant independent of $H$, by
$I_1^H + I_2^H$, where
\[
 I_1^H= \left( \tilde{c}_H \int_{0}^{T} \int_{\R^2}  G^2_{t-s}(x-y)
 \frac{\left(\EE\left[|u_m^H(s,y)-u_m^H(s,z)|^p\right]\right)^{\frac{2}{p}}}{|y-z|^{2-2H}}\, dydzds\right)^{\frac{p}{2}}
\]
and
\[
 I_2^H = \left(  \tilde{c}_H \int_{0}^{T} \int_{\R^2}  \left(\EE\left[| u_m^H(s,z) |^p\right]\right)^{\frac{2}{p}}
 \frac{|G_{t-s}(x-y)-G_{t-s}(x-z)|^2}{|y-z|^{2-2H}} \, dydzds  \right)^{\frac{p}{2}} \\
\]
By the induction hypothesis, the term $I_1^H$ is uniformly bounded in $H$ and $(t,x)$.
Regarding $I_2^H$, using again the induction hypothesis and applying \cite[Prop. 2.8]{BJQ15},
we get
\begin{align*}
I_2^H & \leq  \sup_{H\in K}\sup_{(t,x)\in[0,T]\times \R} \EE\left[| u_m^H(t,x) |^p\right]
 \left(  \tilde{c}_H \int_{0}^{T} \int_{\R^2}   \frac{|G_{t-s}(x-y)-G_{t-s}(x-z)|^2}{|y-z|^{2-2H}}
 \, dydzds  \right)^{\frac{p}{2}}\\
 & \leq C \left(   c_H  \int_{0}^{T}  \int_{\R} |\F G_{t-s}(\xi)|^2 |\xi|^{1-2H} d\xi ds  \right)^{\frac{p}{2}},
\end{align*}
where we recall that the constant $c_H$ is given by
\[
    c_H=\frac{\Gamma(1+2H)\sin(\pi H)}{2\pi}.
\]
Notice that $c_H\leq \frac{1}{2\pi}$, for any $H\in(\frac{1}{4},\frac{1}{2})$.
Moreover, by Lemma \ref{lemma: 3.1}, it holds that
    \begin{equation}
    \label{eq: bound lemma 3.1}
    \int_{0}^{T}  \int_{\R} |\F G_{t-s}(\xi)|^2 |\xi|^{1-2H} d\xi ds \leq  \begin{cases}2^{2H}{C}_{1-2H}\frac{1}{1+2H}T^{1+2H} &  \text{wave equation,} \\
    & \\
    \frac{1}{H}\Gamma(1-H) T^{H}& \text{heat equation.}\end{cases}
    \end{equation}
As explained in Step 1 of the proof of \cite[Thm. 2.8]{paper1}, all constants appearing in \eqref{eq: bound lemma 3.1}
can be bounded uniformly in $H\in K$. This let us conclude that $u^H_{m+1}$ satisfies condition (ii).

It remains to prove that $u^H_{m+1}$ verifies (iii). The
computations follow exactly as in Step 3 of the proof of \cite[Thm.
3.7]{BJQ15}, in such a way that we apply the induction hypothesis,
\cite[Prop. 2.8]{BJQ15} and Lemmas \ref{lemma: 3.1} and \ref{lemma:
D.2}. We omit the details. Nevertheless, we point out why the
presence of the constant $ \tilde{c}_H$ in condition (iii) is
crucial in order to get uniform estimates with respect to $H$.
Precisely, one of the terms appearing in the treatment of the
expression in (iii) for $u^H_{m+1}$ can be bounded by
\begin{equation*}
 A:=  \tilde{c}_H \, C \int_{0}^{t}\int_{\R^2} \frac{G^2_{t-s}(x-y)}{|z|^{2-2H}} \,dy ds
 \int_{0}^{s}\int_{\R} |1-e^{-i\xi z}|^2|\F G_{s-r}(\xi)|^2|\xi|^{1-2H} \,d\xi dr.
\end{equation*}
By Lemma \ref{lemma: D.2}, we have
$$\int_{\R} \frac{|1-e^{-i\xi z}|^2}{|z|^{2-2H}}dz= \frac{2\Gamma(2H+1)\sin(\pi H)}{H(1-2H)}|\xi|^{1-2H}.$$
Hence,
\begin{equation}
\label{eq: last inequality A_2} A \leq  \tilde{c}_H \frac{2
\Gamma(2H+1)\sin(\pi H)}{H (1-2H)} \, C
    \int_{0}^{t}\int_{\R}  G^2_{t-s}(x-y) \,dy ds
    \int_{0}^{s}\int_{\R}    |\F G_{s-r}(\xi)|^2|\xi|^{2(1-2H)} \,d\xi
    dr.
\end{equation}
Note that, by definition of $\tilde{c}_H$ (see Remark \ref{rmk:99}), it holds
$$ \tilde{c}_H \frac{2 \Gamma(2H+1)\sin(\pi H)}{H (1-2H)}=\Gamma(2H+1)\sin(\pi H),$$
and the latter is uniformly bounded for $H\in K$, since it is a
continuous function of $H$. Regarding the integrals in \eqref{eq:
last inequality A_2}, they can be estimated using the explicit
expressions of the fundamental solutions of the wave and heat
equations and applying Lemma \ref{lemma: 3.1}.
\end{proof}

\smallskip

We need to extend condition (ii) in the above proposition to a
uniform estimate with respect to $m\geq 1$. For this, we follow the
arguments of \cite[Section 3.3]{BJQ15}, so we first need the
following result, whose proof follows the same steps of \cite[Thm.
3.8]{BJQ15} and uses analogous arguments as those in Proposition
\ref{prop: uniform estimates for picard iterations}.

\begin{proposition}
    \label{prop: V_n and W_n}
Define, for any $m\geq 0$ and $t\in [0,T]$,
\[
V_m(t):=  \sup_{H\in K}\sup_{x\in\R}
\left(\EE\left[|u^H_m(t,x)-u^H_{m-1}(t,x)|^p\right]\right)^{\frac{2}{p}}
\]
and
\begin{align*}
W_m(t) :=  &  \sup_{H\in K}\sup_{x\in\R} C_H\int_{0}^{t}\int_{\R^2}
G_{t-s}^2(x-y) |y-z|^{2H-2} \\
& \qquad \qquad \times  \left( \EE\left[|u_m^H(s,y)-u^H_{m-1}(s,y) -
u^H_{m}(s,z) + u_{m-1}^H(s,z)|^p\right] \right)^{\frac{2}{p}}\, dy
dz ds.
\end{align*}
Then,
    \begin{equation*}
    V_{m+1}(t)\leq \int_{0}^{t} V_m(s)J_1(t-s)ds +CW_m(t)
    \end{equation*}
    and
    \begin{equation*}
    W_{m+1}(t)\leq \int_{0}^{t} V_m(s)J_2(t-s) ds + \int_{0}^{t} W_m(s)J_1(t-s)ds,
    \end{equation*}
    where $J_1$ and $J_2$ are non-negative integrable functions on $[0,T].$
\end{proposition}

Next, we have the following result on the convergence of the
underlying Picard iteration scheme, which extends \cite[Thm.
3.9]{BJQ15}:
\begin{theorem}
Let $H\in(\frac14,\frac{1}{2})$ and $p\geq 2$. The sequence
$\{u_m^H, \,m\geq 0\}$ of Picard iterations converges in the space
$\chi^p_H$ to a process $u^H$ which is the unique mild solution of
\eqref{eq: mild formulation LMC}. Moreover, it holds:
\begin{equation}
\lim_{m\to \infty} \sup_{H\in K}\sup_{(t,x)\in[0,T]\times \R}
\EE\left[|u_m^H(t,x)-u^H(t,x)|^p\right]=0.
\label{eq:55}
\end{equation}
\end{theorem}
\begin{proof}
As in the proof of \cite[Thm. 3.9]{BJQ15}, we have to check that the
modified definitions of $V_m$ and $W_m$ still work to show that the
Picard iterations converge to the solution $u^H$, uniformly with
respect to $H\in K$. There is no need to check that the solution is
the same as the one found in \cite{BJQ15}, since for any fixed value
of $H$ the norm $\|\cdot\|_{\chi^H}$ is equivalent to the one
defined in \cite[Def. 3.6]{BJQ15}, as we noticed in Remark \ref{rem:
normalizing constant C_H for the norm}.

Set
    $$M_m(t):=V_m(t)+W_m(t)$$
    and
    $$J(t):=C(J_1(t)+J_2(t)).$$
Then, by Proposition \ref{prop: V_n and W_n}, we have
    $$M_{m+1}(t)\leq \int_{0}^{t} (M_m(s)+M_{m-1}(s))J(t-s)ds.$$
The Gr\"onwall type lemma \cite[Lem. 3.10]{BJQ15}) yields
    $$\sum_{m\geq 1} \sup_{H\in K} \|u^H_m-u_{m-1}^H\|_{\chi^p_H}<\infty.$$
This implies that $\{u_m^H\}_{m\geq 0}$ is a Cauchy sequence in
$\chi^p_H$, uniformly with respect to $H\in K$, and so it converges,
uniformly in $H$, to the limit $u^H$, which we already know that
exists and is unique.
\end{proof}

\smallskip

\begin{corollary}
    \label{cor: uniform convergence in L^p}
    Let $H\in(\frac14,\frac{1}{2})$ and $p\geq 2$. Let $K$ be of the
    form described in Proposition \ref{prop: tightness LMC H<1/2}.
    Then, it holds that
    $$\sup_{H\in K} \sup_{m\geq 0} \sup_{(t,x)\in[0,T]\times \R}\EE\left[|u_m^H(t,x)|^p\right]<\infty.$$
\end{corollary}

\smallskip

This corollary, together with the lemmas in the Appendix, allow us
to prove the following result, which is an adaptation of \cite[Prop.
2.2]{holder}. Indeed, as in the preceding result, one just needs to
keep track on the constants depending on $H$.

\begin{proposition}
    \label{prop: condition Q picard iterations}
Let $h_0\in(0,1)$ and $p\geq 2$. Then, for all $|h|\leq h_0$,
\[
\sup_{H\in K} \sup_{(t,x)\in[0,T]\times \R}
\EE\left[|u_m^H(t,x+h)-u_m^H(t,x)|^p\right] \leq C_m|h|^{\eta_1 p}
\]
and
\[
\sup_{H\in K} \sup_{(t,x)\in[0,T\wedge(T-h)]\times \R}
\EE\left[|u_m^H(t+h,x)-u_m^H(t,x)|^p\right] \leq
C_m|h|^{\tilde{\eta}_1 p},
\]
where $\tilde{\eta}_1=\eta_1$ for the wave equation
$\tilde{\eta}_1=\frac{\eta_1}{2}$ for the heat equation. The
constant $C_m$ satisfies
\begin{equation*}
 C_m\leq
C(c(h_0)+\overline{c}(h_0)C_{m-1}),
\end{equation*}
where the functions $c,\overline{c}:\R\to\R$ are non-negative and
$\lim\limits_{h_0\to 0} \overline{c}(h_0)=0$. We define $C_{-1}=0$.
\end{proposition}

Putting together \eqref{eq:55} and Proposition \ref{prop: condition
Q picard iterations}, and taking into account that the sequence
$\{C_m, \, m\geq 0\}$ in the latter result is bounded (see \cite[Thm
1.1]{holder}), we finally have the following:
\begin{proposition}
    \label{prop: holder condition solution}
Let $p\geq 2$.
    There exists $h_0>0$ such that, for every $|h|\leq h_0$, it
    holds:
    \begin{equation*}
    \sup_{H\in K}\sup_{(t,x)\in[0,T]\times \R} \EE\left[|u^H(t,x+h)-u^H(t,x)|^p\right]\leq
    C |h|^{\eta_1 p}
    \end{equation*}
    and
    \begin{equation*}
    \sup_{H\in K}\sup_{(t,x)\in[0,T\wedge (T-h)]\times \R} \EE\left[|u^H(t+h,x)-u^H(t,x)|^p\right]
    \leq C |h|^{{\tilde{\eta}}_1 p},
    \end{equation*}
    where $C$ is a constant depending only on $p$,
     $\tilde{\eta}_1=\eta_1$ for the wave equation and
$\tilde{\eta}_1=\frac{\eta_1}{2}$ for the heat equation.
\end{proposition}

\smallskip

Now, we have all needed ingredients to prove our tightness result
Proposition \ref{prop: tightness LMC H<1/2}.

\smallskip

\noindent {\it{Proof of Proposition \ref{prop: tightness LMC
H<1/2}}}. We will apply Theorem \ref{th: centsov}. First, we notice
that condition (i) in this criterion is clearly satisfied, since
$u^H(0,0)$ is deterministic and independent of $H$.

In order to check (ii) in Theorem \ref{th: centsov}, we apply
Proposition \ref{prop: holder condition solution} and we deduce
that, for any $t,t'\in [0,T]$ and $x,x'\in \R$ such that $|t'-t|<
h_0$ and $|x'-x|<h_0$, it holds:
\begin{equation}
    \label{eq: estimate holder type final}
    \EE\left[|u^H(t',x')-u^H(t,x)|^p\right]\leq C(|t'-t|^{p\tilde{\eta}_1}+|x'-x|^{p\eta_1}).
\end{equation}
One can easily deduce that estimate \eqref{eq: estimate holder type final} holds
for any $t,t'\in [0,T]$ and any $x,x'$ in a compact set.
\qed


\subsection{Tightness in the case $[\frac 12, 1)$} \label{subsec:
tightness H>1/2 LMC}

We aim to prove an analogous tightness result as Proposition
\ref{prop: tightness LMC H<1/2} for the case $H\geq \frac 12$. We
state it in Proposition \ref{prop: tightness LMC H>1/2} below.

Now, we suppose that the limiting exponent $H_0\in [\frac 12 ,1)$,
so whenever $H_n\to H_0$ we can suppose without loss of generality
that $H_n\in K$, where $K$ is of the form $[\eta_1,\eta_2]$, with
$\eta_1,\eta_2\in [\frac 12,1)$ and $\eta_1\leq \eta_2$. As we
already observed at the beginning of Section \ref{subsec: tightness
H<1/2 LMC}, if we prove the tightness of the family of laws of
$\{u^{H},\, H\in K\}$ also for $K$ of the form considered here, this
will include also the case in which $H_0=\frac 12$ and $H_n\to H_0$
either from above or from below.

The following tightness result will be proved directly, i.e. without
going through the corresponding Picard iteration scheme. This is
because the Burkholder-Davies-Gundy type inequality \eqref{eq:
Burkholder H>1/2} is more practical than its {\it{rough}}
counterpart \eqref{eq: Burkholder H<1/2}.

\begin{proposition}
    \label{prop: tightness LMC H>1/2}
    Let  $\,\mathcal{U}_K:=\{u^{H},\, H\in K\}$ be the family of solutions of \eqref{eq: mild formulation LMC}, where $K$ is of the form $[\eta_1,\eta_2]$, with $\eta_1,\eta_2\in [\frac 12, 1)$ and $\eta_1\leq \eta_2$. Then, the family $\mathcal{U}_K$ is tight in $\C([0,T]\times \R)$, endowed with the metric of uniform convergence on compact sets.
\end{proposition}

\begin{proof}
    We will apply again Theorem \ref{th: centsov}. We split the proof in three
    steps.
\smallskip

\noindent {\it{Step 1}}: We show the uniform estimate
    \begin{equation}
    \label{eq: Lp bound uniform in H}
    \sup_{H\in [\eta_1,\eta_2]}\sup_{(t,x)\in[0,T]\times \R}\EE\left[|u^H(t,x)|^p\right]<\infty.
    \end{equation}
We have
\begin{equation*}
\EE\left[|u^H(t,x)|^p\right]
    \leq  C \left( 1+ \EE\left[\left|\int_{0}^{t}\int_{\R} G_{t-s}(x-y)u^H(s,y)W^H(ds,dy)\right|^p
    \right]\right).
\end{equation*}
By Theorem \ref{teo: Integrability condition H>1/2}, we obtain that
the expectation in the right hand-side above can be bounded, up to
some positive constant, by
\begin{equation}
\label{eq:234}
c_H (\nu_{t,H})^{\frac{p}{2}-1} \int_{0}^{t} \sup_{H\in
[\eta_1,\eta_2]} \sup_{x\in\R} \EE\left[|u^H(s,x)|^p\right]
 \int_{\R} |\F G_{t-s}(x-\cdot)(\xi)|^2|\xi|^{1-2H}\, d\xi ds,
\end{equation}
where $\nu_{t,H}$ is defined by
\begin{equation*}
\nu_{t,H} =  c_H\int_{0}^{t}\int_{\R} |\F
G_{s}(\xi)|^2|\xi|^{1-2H}d\xi ds.
\end{equation*}
We recall that $c_H=\frac{\Gamma(2H+1)\sin(\pi H)}{2\pi}$, which is
bounded by $\frac{1}{\pi}$, for all $H$. Moreover, by Lemma
\ref{lemma: 3.1}, it holds that
\begin{equation*}
 \sup_{H\in[\eta_1,\eta_2]}\sup_{t\in [0,T]} \nu_{t,H} < \infty.
\end{equation*}
Note that this holds for both wave and heat equations. On the other
hand, regarding the integral in $d\xi$  in \eqref{eq:234}, we can argue as follows. In the case of the wave
equation, we have
\begin{align*}
\int_{\R} |\F G_{t-s}(x-\cdot)(\xi)|^2|\xi|^{1-2H}d\xi & =
    2\int_{0}^{\infty} \frac{\sin^2((t-s)\xi)}{\xi^{1+2H}}d\xi \nonumber \\
    & = 2(t-s)^{2H} 2^{2H-1} C_{1-2H} \\
    & \leq  T^{2H}2^{2H}C_{1-2H},
\end{align*}
where the constant $C_{1-2H}$ is the same one appearing in Lemma
\ref{lemma: 3.1}. As showed in the proof of \cite[Thm. 2.8]{paper1},
$C_{1-2H}$ defines a continuous function with respect to
$H\in(0,1)$, so it can be bounded by a constant when
$H\in[\eta_1,\eta_2]$. Thus, for the wave equation we can conclude
that
\begin{equation*}
    \sup_{H\in [\eta_1,\eta_2]}\sup_{x\in\R} \EE\left[|u^H(t,x)|^p\right] \leq
    C \left(1 +  \int_{0}^{t} \sup_{H\in [\eta_1,\eta_2]} \sup_{x\in\R} \EE\left[|u^H(s,x)|^p
    \right] ds \right).
\end{equation*}
Hence, Gr\"onwall lemma implies \eqref{eq: Lp bound uniform in H}.

In the case of the heat equation, we have
\begin{align*}
\int_{\R}|\F G_{t-s}(x-\cdot)(\xi)|^2|\xi|^{1-2H}d\xi &
= \int_{\R}e^{-(t-s)|\xi|^2}|\xi|^{1-2H}d\xi  \\
& = \frac 12 (t-s)^{H-1}\int_{0}^{\infty} e^{-y} y^{-H}  dy \\
& = \Gamma(1-H) (t-s)^{H-1}.
\end{align*}
Observe that, for all $H\in[\eta_1,\eta_2]$, it holds $\Gamma(1-H)
(t-s)^{H-1}\leq g(t-s)$, where
$${g}(r):=\Gamma(1-\eta_2) \begin{cases}
    r^{\eta_1-1}, & r<1 \\
    1, & r>1.
\end{cases}
$$
Therefore,
\begin{equation*}
    \sup_{H\in [\eta_1,\eta_2]}\sup_{x\in\R} \EE\left[|u^H(t,x)|^p\right] \leq
    C \left(1 +  \int_{0}^{t} \sup_{H\in [\eta_1,\eta_2]} \sup_{x\in\R} \EE\left[|u^H(s,x)|^p
    \right] g(t-s)ds \right).
\end{equation*}
The Gr\"onwall type lemma proved in \cite[Lem. 15]{Dalang} let us
conclude that \eqref{eq: Lp bound uniform in H} is also fulfilled in
the case of the heat equation.

\smallskip

\noindent {\it{Step 2}}: In this part of the proof, we deal with the
moments of the space increments of the solution $u^H$. Precisely,
owing to Theorem \ref{teo: Integrability condition H>1/2}, the
estimate \eqref{eq: Lp bound uniform in H} and Lemma \ref{lemma:
3.4}, we can infer that, for all $p\geq 2$ and $|h|\leq 1$,
\begin{align*}
& \EE\left[|u^H(t,x+h)- u^H(t,x)|^p\right]\\
& \qquad
 \leq C \, c_H^{\frac p2}
\left(\int_{0}^{t}
 \int_{\R} \left|\F\left( G_{t-s}(x-\cdot)+G_{t-s}(x+h-\cdot) \right)(\xi)\right|^2
 |\xi|^{1-2H}\, d\xi ds\right)^{\frac p2}\\
& \qquad = C \, c_H^{\frac p2} \left( \int_{0}^{t} \int_{\R}
(1-\cos(h\xi))|\F G_s(\xi)|^2 |\xi|^{1-2H} \, d\xi ds
\right)^{\frac{p}{2}} \\
& \qquad \leq C \, \tilde{C}_H^{\frac p2}  |h|^{Hp}.
\end{align*}
The constant $\tilde{C}_H$ is the same appearing in \cite[Lem.
3.4]{BJQ15}, and it is given by
$$\tilde{C}_H:=\int_{\R}(1-\cos(\theta))|\theta|^{-1-2H}d\theta<\frac{1}{H}+\frac{1}{1-H}\leq C,$$
provided that $H\in[\eta_1,\eta_2].$  Thus, we have proved that
$$\sup_{H\in [\eta_1,\eta_2]}\sup_{(t,x)\in[0,T]\times \R}
\EE\left[|u^H(t,x+h)-u^H(t,x)|^p\right] \leq C|h|^{\eta_1 p}.$$

\noindent {\it{Step 3}}: Here, we aim to prove that, for any $p\geq
2$ and $|h|<1$,
\begin{equation}
 \sup_{H\in
[\eta_1,\eta_2]}\sup_{(t,x)\in[0 \vee(-h),T\wedge(T-h)]\times \R}
\EE\left[|u^H(t+h,x)-u^H(t,x)|^p\right]
 \leq \begin{cases}
  C |h|^{\eta_1 p} & \text{ wave equation,}\\
  C|h|^{\frac{\eta_1  }{2}p}  & \text{ heat equation.}
 \end{cases}
 \label{eq:123}
\end{equation}
Assume that  $h>0$ (the case $h<0$ is completely analogous). Then,
\[
\E{|u^H(t+h,x)-  u^H(t,x)|^p}\leq C(B_1+B_2),
\]
where
\[
B_1 := \EE \left[\left|\int_{0}^{t}\int_{\R}
[G_{t+h-s}(x-y)-G_{t-s}(x-y)]u^H(s,y)W^H(ds,dy) \right|^p\right],
\]
\[
B_2 := \EE\left[\left|\int_{t}^{t+h}\int_{\R}
G_{t+h-s}(x-y)u^H(s,y)W^H(ds,dy)\right|^p\right].
\]
Theorem \ref{teo: Integrability condition H>1/2}, \eqref{eq: Lp
bound uniform in H} and Lemma \ref{lemma: 3.5} yield
\begin{align}
B_1 & \leq C \, c_H^{\frac p2} \left(\int_{0}^{t}\int_{\R} \left|\F
\left( G_{t+h-s}(x-\cdot)-G_{t-s}(x-\cdot) \right)(\xi)
\right|^2|\xi|^{1-2H}d\xi ds\right)^{\frac p2}\nonumber \\
& \leq C \, c_H^{\frac p2}\left( \int_{0}^{T} \int_{\R} \left|\F
G_{s+h}(\xi)-\F G_{s}(\xi) \right|^2|\xi|^{1-2H} d\xi ds
\right)^{\frac p2}\nonumber \\
& \leq C \, \begin{cases}
    |h|^{H p}, & \text{ wave equation,} \\
    |h|^{\frac{H}{2}p}, & \text{ heat equation.}
    \end{cases}
\label{eq:68}
\end{align}

Regarding the term $B_2$, we can argue as before but we apply Lemma
\ref{lemma: 3.1}. Indeed, we have that
\begin{align}
B_2 & \leq C\, c_H^{\frac p2} \left(\int_{t}^{t+h}\int_{\R} |\F
G_{t+h-s}(x-\cdot)(\xi)|^2|\xi|^{1-2H} d\xi ds\right)^{\frac p2} \nonumber \\
& = C\, c_H^{\frac p2} \int_{0}^{h} \int_\R |\F
G_{s}(\xi)|^2|\xi|^{1-2H} d\xi ds \nonumber \\
& \leq C
\begin{cases}
     |h|^{\frac{1+2H}{2}p}, & \text{ wave equation,} \\
     |h|^{\frac{H}{2}p}, & \text{ heat equation.} \\
    \end{cases}
\label{eq:69}
\end{align}
Putting together \eqref{eq:68} and \eqref{eq:69}, and taking into
account that $H\in [\eta_1,\eta_2]$, we end up with \eqref{eq:123}.

\smallskip

Finally, the results in Steps 2 and 3 let us conclude that, for any
$t,t'\in [0,T]$ and $x,x'$ in a compact of $\R$, we have
\[
\EE\left[|u^H(t',x')-u^H(t,x)|^p\right]\leq C  \begin{cases}
    |t'-t|^{\eta_1 p} + |x'-x|^{\eta_1 p}, &  \text{ wave equation,}\\
    |t'-t|^{\frac{\eta_1}{2} p}+ |x'-x|^{\eta_1 p}, & \text{ heat equation} \\
    \end{cases}.
\]
Thus, it suffices to take $p>\frac{4}{\eta_1}$ for the heat equation
and $p>\frac{2}{\eta_1}$ for the wave equation to be able to apply
the tightness criterion Theorem \ref{th: centsov}.
\end{proof}

\medskip

The following result extends Corollary \ref{cor: uniform convergence
in L^p} to the case $H\geq\frac{1}{2}$. Its proof is very similar to
that of \cite[Thm. 13]{Dalang}, and the
    terms that need to be estimated uniformly with respect to $H$ are
    completely analogous as those appearing in Step 1 of the proof
    of the above Proposition \ref{prop: tightness LMC H>1/2}.

\begin{lemma}
    \label{lemma: uniform convergence in L^p}
Let $H\geq \frac{1}{2}$ and  $\{u^H_m, \, m\geq 0\}$ be the sequence
of Picard iterations corresponding to the mild formulation
\eqref{eq: mild formulation LMC}. Then, for any $p\geq 2$, $u_m^H$
converges in $L^p(\Omega)$ to the solution $u^H$ uniformly with
respect to $H\in K$, i.e.
$$\lim_{m\to \infty} \sup_{H\in K} \sup_{(t,x)\in [0,T]\times \R}
\EE\left[|u_m^H(t,x)-u^H(t,x)|^p\right] = 0$$
\end{lemma}


\section{Identification of the limit}
\label{subsec: limit identification LMC}

Let $H_0\in (\frac 14 ,1 )$ and $\{ H_n,\, n\geq 1\}$ be any
sequence such that $H_n\to H_0$, as $n\to\infty$.
We may assume that there exists a compact set $K\subset (\frac 14,1)$ such that
$H_n\in K$, for all $n\geq 1$.
The tightness
results proved in Propositions \ref{prop: tightness LMC H<1/2} and
\ref{prop: tightness LMC H>1/2} imply that there exists a
subsequence $\{H_{n_k},\, k\geq 1\}$ such that $\{u^{H_{n_k}}, \,
k\geq 1\}$ converges in law in the space $\mathcal{C}([0,T]\times
\R)$  of contiuous functions. This section is devoted to prove that
the limit law is the distribution of $u^{H_0}$.

Our strategy can be summarized as follows. We will verify that the finite dimensional
distributions of $u^{H_n} $ converge to
those of $u^{H_0}$ (see \cite[Thm. 2.6]{billingsley}). For this, it suffices to prove that,
for any fixed $(t,x)\in [0,T]\times \R$, $u^{H_n}(t,x)$ converges to $u^{H_0}(t,x)$ in
$L^2(\Omega)$. This can be done thanks to the fact that the whole family of noises $\{W^H,\, H\in
(0,1)\}$ can be defined on a single probability space (see Section \ref{sec:spectral}). In order
to prove the above $L^2(\Omega)$-convergence, we will check the same convergence
for any of the corresponding Picard iterates, that is, for any $m\geq 1$, we show that
$u^{H_n}_m(t,x)\to
u^{H_0}_m(t,x)$ in $L^2(\Omega)$, as $n\to\infty$, and we will take into account that the
Picard iteration scheme converges to the soution uniformly with respect to the Hurst index $H$.
At this point, we recall (invoking Theorem \ref{teo: equivalence ito-skorohod
solution}) that any Picard iterate admits the following Wiener chaos expansion:
$$u_m^{H_n}(t,x)=\sum_{j=0}^{m} I^{H_n}_j(g_j(\cdot,t,x)),$$
where the latter is a finite sum of multiple Wiener integrals of
order up to $m$ and the kernels $g_j$ are given by \eqref{eq: g_n}.
Therefore, it will be sufficient to prove the $L^2(\Omega)$-convergence, as $n\to \infty$, of any
of the above multiple Wiener integrals, for which we will make use of the representation result
given in Theorem \ref{teo: representation multiple integral}.

Here is the main result of the section:

\begin{theorem}
    \label{prop: finite dimensional convergence LMC}
    Let $H_0\in(\frac{1}{4},1)$ and $\{ H_n,\, n\geq 1\}$ be any sequence
such that $H_n\to H_0$, as $n\to\infty$.
Let $u^{H_n}_n$ and $u^{H_0}$ be the solutions of \eqref{eq: mild formulation LMC} corresponding
the Hurst parameters $H_n$ and $H_0$, respectively.
Then, the finite dimensional distributions of $u^{H_n}$ converge to those of $u^{H_0}$, as $n\to\infty$.
\end{theorem}

\begin{proof}
We split the proof in three steps.

\smallskip

\noindent {\it{Step 1:}} To start with, we recall that, owing to
Corollary \ref{cor: uniform convergence in L^p} and Lemma
\ref{lemma: uniform convergence in L^p} in the particular case
$p=2$, we have:
\begin{equation}
\lim_{m\to \infty} \sup_{H\in K} \sup_{(t,x)\in[0,T]\times \R} \EE\left[|u^H_m(t,x)-u^H(t,x)|^2\right] = 0,
\label{eq:70}
\end{equation}
where $u_m^H$ denotes the associated $m$th Picard iterate.

As we already explained, in order to assure the statement's validity it is sufficient to show the
following pointwise convergence in $L^2(\Omega)$: for any fixed $(t,x)\in [0,T]\times \R$, it holds
    \begin{equation*}
    \lim_{n\to\infty} \EE\left[|u^{H_n}(t,x)-u^{H_0}(t,x)|^2\right] = 0.
    \end{equation*}
Note that
    \begin{align*}
    & \EE\left[|u^{H_n}(t,x)-u^{H_0}(t,x)|^2\right] \\
    & \quad  \leq C\left(    \EE\left[|u^{H_n}(t,x)-u^{H_n}_m(t,x)|^2\right]
    + \EE\left[|u^{H_n}_m(t,x)-u^{H_0}_m(t,x)|^2\right]
    +   \EE\left[|u^{H_0}_m(t,x)-u^{H_0}(t,x)|^2\right] \right) \\
    & \quad =:I_1(m,n)+I_2(m,n)+I_3(m).
    \end{align*}
By \eqref{eq:70}, we can infer that, for any $\varepsilon>0$, we can choose $m_0$ big enough such that,
for every $m\geq m_0$, we have
    $$\sup_{n\geq 1} \left\{I_1(n,m)+I_3(m)\right\} < \varepsilon.$$
Thus, we are left to show that $I_2(m_0,n)$ tends to zero as $n\to \infty$. This means, in particular,
that the $m_0$-th Picard iterate is continuous in $L^2(\Omega)$, with respect to $H$.

Theorem \ref{teo: equivalence ito-skorohod solution} implies that, for any $H\in (\frac{1}{4},1)$,
$u_{m_0}^{H}$ has the Wiener chaos expansion
    $$u_{m_0}^{H}(t,x)=\sum_{j=0}^{m_0} I^H_j(g_j(\cdot,t,x)),$$
where the functions $g_j$ are defined by \eqref{eq: g_n}. Hence, in order to check that
$I_2(m_0,n)$ tends to zero, it is enough to show that, for any $j=1,\dots,m_0$,
$I_j^{H_n}(g_j(\cdot,t,x))$ converges to $I_j^{H_0}(g_j(\cdot,t,x))$ in $L^2(\Omega)$, as $n\to \infty$. Indeed, by
Theorem \ref{teo: representation multiple integral} we have that
    \begin{align*}
     & I^{H_n}_j(g_j(\cdot,t,x))- I^{H_0}_j(g_j(\cdot,t,x)) \\
     &\qquad = \int_{\{[0,T]\times\R\}^j} \left(c_{H_n}^j |\xi_1|^{\frac{1}{2}-H_n}\cdots |\xi_j|^{\frac{1}{2}-H_n}-
    c_{H_0}^j |\xi_1|^{\frac{1}{2}-H_0}\cdots |\xi_j|^{\frac{1}{2}-H_0}\right) \\
     & \qquad \qquad \times \F g_j(t_1,\cdot,\dots, t_j,\cdot,t,x)(\xi_1,\dots,\xi_n)
     \tilde{W}(dt_1,d\xi_1)\cdots \tilde{W}(dt_j,d\xi_j).
    \end{align*}
Hence
\begin{align*}
    & \EE \left[\left|I^{H_n}_j(g_j(\cdot,t,x))-  I^{H_0}_j(g_j(\cdot,t,x))\right|^2\right] \\
     & \qquad = \int_{\{[0,T]\times\R\}^j} \left|c_{H_n}^j|\xi_1|^{\frac{1}{2}-H_n}\cdots |\xi_j|^{\frac{1}{2}-H_n}-
     c_{H_0}^j |\xi_1|^{\frac{1}{2}-H_0}\cdots |\xi_j|^{\frac{1}{2}-H_0}\right|^2   \\
    & \qquad \qquad \times |\F g_j(t_1,\cdot,\dots, t_j,\cdot,t,x)(\xi_1,\dots,\xi_n)|^2 d\xi_1\cdots d\xi_j \, dt_1\cdots dt_j . \\
\end{align*}
    We show that the last integral converges to $0$ when $n\to\infty$. To do this, we have to compute explicitly
    the Fourier transform appearing in the above expression. Precisely, as detailed in \cite[p. 10]{intermittency},
    we have
\begin{align*}
    & \F g_j(t_1,\cdot,\dots, t_j, \cdot,t,x)(\xi_1, \dots,\xi_j) \\
    & \qquad = \eta e^{-i(\xi_1+\cdots+\xi_j)x} \, \overline{ \F G_{t_2-t_1}(\xi_1)}  \,
    \overline{ \F G_{t_3-t_2}(\xi_1+\xi_2) } \cdots
    \overline{ \F G_{t-t_j}(\xi_1+\cdots+\xi_j) } \, 1_{\{0<t_1<\cdots <t_j<t\}}
 \end{align*}
Therefore, making the change of variables $\eta_{\ell}:=\xi_1+\cdots + \xi_\ell$, for $\ell=1,\dots,j$, we end up with
\begin{align*}
    & \EE\left[|I^{H_n}_j(g_j(\cdot,t,x))-  I^{H_0}_j(g_j(\cdot,t,x))|^2\right] \nonumber \\
    & \quad \leq  \int_{T_j(t)}\int_{\R^j}   \eta   \prod_{\ell=1}^j |\F G_{t_{\ell+1}-t_\ell} (\eta_{\ell})|^2
    \, \left| c_{H_n}^j  |\eta_1|^{\frac{1}{2}-H_n}|\eta_2-\eta_1|^{\frac{1}{2}-H_n} \cdots |\eta_{j}-\eta_{j-1}|^{\frac{1}{2}-H_n}\right. \nonumber \\
    & \qquad  - \left. c_{H_0}^j  |\eta_1|^{\frac{1}{2}-H_0}|\eta_2-\eta_1|^{\frac{1}{2}-H_0} \cdots |\eta_{j}-\eta_{j-1}|^{\frac{1}{2}-H_0} \right|^2 d\xi_1\cdots d\xi_j \, dt_1\cdots dt_j,
\end{align*}
where $T_j(t):=\{ (t_1,\dots,t_j),\, 0<t_1<\cdots<t_j<t \}$. We wish
to prove that the latter integral converges to 0 as $n\to\infty$.
For this, we will apply the Dominated convergence theorem. Note that
the integrand clearly converges to 0 pointwise on $T_j(t)\times
\R^j$. Indeed, the constant $c_H$ (see \eqref{eq:103}) defines a
continuous function of $H\in (0,1)$. Now, we proceed to bound the
integrand by an integrable function. First, we note that the
integrand can be bounded, up to some positive constant, by
\begin{align*}
    & \prod_{\ell=1}^j |\F G_{t_{\ell+1}-t_\ell} (\eta_{\ell})|^2 \,  \left(  c_{H_n}^{2j} |\eta_1|^{1-2H_n}|\eta_2-\eta_1|^{1-2H_n} \cdots |\eta_{j}-\eta_{j-1}|^{1-2H_n}\right. \\
    & \qquad\qquad + \left. c_{H_0}^{2j} |\eta_1|^{1-2H_0}|\eta_2-\eta_1|^{1-2H_0} \cdots |\eta_{j}-\eta_{j-1}|^{1-2H_0}  \right).
\end{align*}
The two resulting terms in the above sum are of the same type,
except the fact that the first one depends on $n$ while the second
does not, and they are equivalent to the integrands studied in
\cite[p. 11-13]{intermittency} (only in the case of wave equation
with $H\in (\frac{1}{4},\frac{1}{2})$). From now on, we will only
consider the term of the integrand function that depends on $n$; the
integrability of the other term will be an immediate consequence of
the treatment of the first one.

Hence, we will find a suitable estimate for the term
\begin{equation}
    \label{eq: integrand dim finale}
           |\eta_1|^{1-2H_n}|\eta_2-\eta_1|^{1-2H_n} \cdots |\eta_{j}-\eta_{j-1}|^{1-2H_n} \, \prod_{\ell=1}^j |\F G_{t_{\ell+1}-t_\ell} (\eta_{\ell})|^2.
\end{equation}
Notice that we have bounded $c_{H_n}$ by a constant, since we may
assume that all $H_n$ are included in a compact set of $(\frac14,1)$. We
distinguish the cases $H_n<\frac12$ and $H_n\geq \frac12$.

\smallskip

\noindent {\it{Step 2:}} In the case $H_n<\frac12$, we use the
following fact: whenever $H\in (0,\frac{1}{2})$, we have
    \begin{equation*}
        \prod_{\ell=2}^{j} |\eta_\ell-\eta_{\ell-1}|^{1-2H}\leq \sum_{\alpha\in D_j}\prod_{\ell=1}^{j} |\eta_{\ell}|^{\alpha_\ell},
    \end{equation*}
    where $D_j$ is a set with cardinality $2^{j-1}$ and its elements are multi-indices $\alpha=(\alpha_1,\dots, \alpha_j)$ whose component's sum equals to $(j-1)(1-2H)$ and satisfy
    $$\alpha_1\in\{0,1-2H\}, \text{ and } \alpha_\ell\in\{ 0, 1-2H, 2(1-2H) \}, \text{ for } \ell=2,\dots, j.$$
When $H=H_n$, the corresponding $\alpha_{\ell}$ will be denoted by $\alpha_{\ell,n}$. Thus, the integrand  \eqref{eq: integrand dim finale} may be bounded by
\begin{equation}
     |\eta_1|^{1-2H_n}\left(\sum_{\alpha\in D_j}\prod_{\ell=1}^{j} |\eta_{\ell}|^{\alpha_{\ell,n}}\right)\left(\prod_{\ell=1}^{j}  \left|\F G_{t_\ell-t_{\ell-1}}(\eta_\ell) \right|^2 \right)
\label{eq:72}
\end{equation}
Let $\beta:=\min_{n\geq 1} H_n>1/4$ and define the functions $f_0,f_1,f_2:\R_+\to \R_+$ as follows: $f_0(r)=1$ and
    $$f_1(r)=\begin{cases}
    r^{1-2\beta}, & r\geq 1, \\
    1, & r < 1,
    \end{cases}$$
    $$f_2(r)=\begin{cases}
    r^{2(1-2\beta)}, & r\geq 1, \\
    1, & r < 1.
    \end{cases}$$
We also set, for every $\alpha_{\ell,n}$,
    $$N(\alpha_{\ell,n}):=
    \begin{cases}
    0, &  \alpha_{\ell,n} =0, \\
    1, & \alpha_{\ell,n}=1-2H_n, \\
    2, & \alpha_{\ell,n}=2(1-2H_n). \\
    \end{cases}$$
Then, we have the following estimate for the term \eqref{eq:72}:
    \begin{align*}
    & |\eta_1|^{1-2H_n}\left(\sum_{\alpha\in D_j}\prod_{\ell=1}^{j} |\eta_{\ell}|^{\alpha_{\ell,n}}\right)\left(\prod_{\ell=1}^{j}  \left|\F G_{t_\ell-t_{\ell-1}}(\eta_\ell) \right|^2 \right)\nonumber \\
    & \qquad \leq   f_1(|\eta_1|) \left(\sum_{\alpha\in D_j}\prod_{\ell=1}^{j} f_{N(\alpha_{\ell,n})  }(|\eta_{\ell}|)\right) \left( \prod_{\ell=1}^{j}  \left|\F G_{t_\ell-t_{\ell-1}}(\eta_\ell) \right|^2 \right).
    \end{align*}
We have to prove that this function is integrable. To check this
last fact, it is sufficient to show it for a single integrand of the
form
    \begin{equation*}
    \prod_{\ell=1}^{j}  \Big|\F G_{t_\ell-t_{\ell-1}}(\eta_\ell) \Big|^2  |\eta_1|^{\beta} \prod_{\ell=1}^{j} |\eta_{\ell}|^{\alpha_{\ell}}
    \end{equation*}
where, now, $\alpha_j$ does not take values in a discrete
set, but they satisfy the weaker constraints:
$$\alpha_1\in K_1\subset [0,1/2), \text{ and } \alpha_\ell\in K_2 \subset [0,1), \text{ for } \ell=2,\dots, j,$$
where $K_1=[0, 1-2\min_{n\geq 1} H_n]$ and $K_2=[0,2(1-2\min_{n\geq
1} H_n)]$ (we are assuming implicitly that $\min_{n\geq 1} H_n
<\frac{1}{2}$; if this is not the case, then the entire sequence
falls in the case $H_n\geq \frac{1}{2}$, which will be studied
afterwards). It is important to notice that the sets $K_1,K_2$ do
not depend on $n$. The fact that $1-2\min_{n\geq 1} H_n<\frac{1}{2}$
and $2(1-2\min_{n\geq 1} H_n)<1$ turns out to be crucial for our
estimates.

Thus, we want to prove that
    \begin{equation}
    \label{eq: integrand3 dim finale}
    \int_{T_j(t)} \left(  \int_{\R} |\F G_{t_2-t_1}(\eta_1)|^2 |\eta_1|^{\beta +\alpha_1} d\eta_1 \right)
     \prod_{\ell=2}^j  \left( \int_{\R} |\F G_{t_{\ell +1}-t_\ell}(\eta_{\ell})|^2 |\eta_\ell|^{\alpha_{\ell}}  d\eta_{\ell}\right) dt_1 \cdots dt_j <\infty.
    \end{equation}
At this point, we have to consider separately the case of the wave
equation case from that of the heat equation. It holds that, for any
$\gamma\in (-1,1)$ (see the proof of Proposition \ref{prop: tightness LMC H>1/2}):
    \begin{equation*}
    \begin{split}
    & \int_{\R} |\F G_t(\xi)|^2|\xi|^{\gamma} d\xi \leq  \begin{cases}
    C'_\gamma (2-\gamma) t^{1-\gamma}, &  \text{ wave equation,}\\
    C''_{\gamma} \frac{1-\gamma}{2} t^{-\frac{(\gamma+1)}{2}}, & \text{ heat equation.}
    \end{cases}
    \end{split}
    \end{equation*}
    We recall that the constants $C'_\gamma$ and $C''_\gamma$ are continuous with respect to $\gamma\in (-1,1)$.
    We will apply the above estimate with $\gamma=1-2H$ and $\gamma=2(1-2H)$, and still we can bound them uniformly with respect to $H\in K \subset (\frac{1}{4},\frac{1}{2}]$, with $K$ compact.
Hence, for the heat equation, the integral in \eqref{eq: integrand3 dim finale} can be estimated by
    \begin{equation*}
    \int_{T_j(t)} (t_2-t_1)^{\frac{-\beta-\alpha_1}{2}} \prod_{\ell=2}^{j} (t_{\ell+1}-t_\ell)^{\frac{-\alpha_\ell-1}{2}} dt_1\cdots dt_j,
    \end{equation*}
    which is finite because all exponents are strictly greater than $-1$.     For the wave equation, we end up with
    \begin{equation*}
    \int_{T_j(t)} (t_2-t_1)^{1-\beta-\alpha_1} \prod_{\ell=2}^{j} (t_{\ell+1}-t_\ell)^{1-\alpha_{\ell}} dt_1\cdots dt_j,
    \end{equation*}
    which is also finite since all exponents are even greater than $0$. This concludes the proof in the case $H\in(\frac{1}{4},\frac{1}{2}]$.

    \smallskip

    \noindent {\it{Step 3:}}
    Let us now go back to expression \eqref{eq: integrand dim finale}, where we resettle the variables $\xi_\ell$ by means of the change of variables $\xi_\ell=\eta_\ell-\eta_{\ell-1}$. That is,
    we aim to bound the following term:
    \begin{equation}
    \label{eq: integrand4 dim finale}
        |\xi_1|^{1-2H_n}\cdots |\xi_j|^{1-2H_n}  \prod_{\ell=1}^j  \left|\F G_{t_{\ell+1}-t_\ell}(\xi_1+\cdots+\xi_\ell )\right|^2,
    \end{equation}
    where we assume that $H_n\in [\frac{1}{2},1)$. Here, the fact that $1-2H_n\leq 0$ helps us. Indeed, we can define the bounding function in a quite straightforward way:
    $$g(r):=\begin{cases}
    1, &  r\geq 1,  \\
    r^{1-2(\max_{n\geq 1} H_n)}, & r<1. \\
    \end{cases}$$
    Clearly, the integrand function in \eqref{eq: integrand4 dim finale} is bounded, for any $n\geq 1$, by
    $$  g(|\xi_1|)\cdots g(|\xi_j|) \, \prod_{\ell=1}^j  \left|\F G_{t_{\ell+1}-t_\ell}(\xi_1+\cdots+\xi_\ell )\right|^2.$$
    We check that this upper bound function is integrable, namely
    \begin{align}
    \label{eq: integrand5 dim finale}
    & \int_{T_{j-1}(t_j)} \int_{\R^{j-1}}  \prod_{\ell=1}^{j-1} \left| \F G_{t_{\ell+1}-t_\ell}(\xi_1+\cdots+\xi_\ell) \right|^2 g(|\xi_\ell|) \nonumber \\
    &\qquad \times \left(\int_{t_{j-1}}^{t} \int_{\R} \left| \F G_{t-t_j}(\xi_1+\cdots+\xi_j) \right|^2g(|\xi_j|)d\xi_j dt_j \right) d\xi_1\cdots d\xi_{j-1} dt_1\cdots dt_{j-1}
    <\infty.
    \end{align}
    We have that
    \begin{align*}
     & \int_{t_{j-1}}^{t} \int_{\R}  \left| \F G_{t-t_j}(\xi_1+\cdots+\xi_j) \right|^2g(|\xi_j|)d\xi_j dt_j \\
     & \qquad =      \int_{t_{j-1}}^{t} \int_{|\xi_j|>1}  \left| \F G_{t-t_j}(\xi_1+\cdots+\xi_j) \right|^2d\xi_j dt_j  \\
    & \qquad \quad +   \int_{t_{j-1}}^{t} \int_{|\xi_j|\leq1}  \left| \F G_{t-t_j}(\xi_1+\cdots+\xi_j) \right|^2 \, |\xi_j|^{1-2\min_{n\geq 1} H_n}d\xi_j dt_j.
    \end{align*}
    We do the computations separately for the  wave and heat equations. To start with, in the case of the wave equation, it clearly holds that
    $$|\F G_{t}(\xi)|=\left|\frac{\sin(t|\xi|)}{|\xi|}\right|\leq
    t,$$
    for all $(t,x)\in [0,T]\times \R$. Thus, we have
    \begin{align*}
     & \int_{t_{j-1}}^{t} \int_{|\xi_j|\leq1}   \left| \F G_{t-t_j}(\xi_1+\cdots+\xi_j) \right|^2|\xi_j|^{1-2\min_{n\geq 1} H_n}d\xi_j dt_j \\
    & \qquad \leq   \int_{t_{j-1}}^{t} \int_{|\xi_j|\leq1} |t-t_j|^2|\xi_j|^{1-2\min_{n\geq 1} H_n}d\xi_j dt_j \\
    & \qquad \leq \frac{CT^3}{1-\min_{n\geq 1} H_n}  <\infty,
    \end{align*}
    and
    \begin{align*}
     & \int_{t_{j-1}}^{t} \int_{|\xi_j|> 1}   \left| \F G_{t-t_j}(\xi_1+\cdots+\xi_j) \right|^2|\xi_j|^{1-2\min_{n\geq 1} H_n} \\
    & \qquad \leq   \int_{t_{j-1}}^{t} \int_{\R}  \frac{\sin^2\left[ (t-t_j) |\xi_1+\cdots+\xi_j| \right]}{|\xi_1+\cdots+\xi_j|^2} d\xi_j dt_j \\
    & \qquad \leq C \int_{t_{j-1}}^t (t-t_j)dt_j<\infty,
    \end{align*}
    since $\int_\R\frac{\sin^2 (t|x|)}{|x|^2}dx=\pi t$.
Therefore, we have got rid of the integral with respect to $d\xi_j dt_j$ in \eqref{eq: integrand5 dim finale}.
Iterating this procedure one proves that the whole integral \eqref{eq: integrand5 dim finale} is finite.

 It remains to prove the analogous result for the heat equation. Here, we have
    $$|\F G_{t}(\xi)|=e^{-\frac{t|\xi|^2}{2}}\leq 1,
    $$
for all $(t,x)\in [0,T]\times \R$. Thus,
\begin{align*}
    & \int_{t_j}^{t} \int_{|\xi_j|\leq1}   \left| \F G_{t-t_j}(\xi_1+\cdots+\xi_j) \right|^2|\xi_j|^{1-2\min_{n\geq 1} H_n}d\xi_j dt_j \\
    &\qquad \leq    \int_{t_j}^{t} \int_{|\xi_j|\leq1} |\xi_j|^{1-2\min_{n\geq 1} H_n}d\xi_j dt_j \\
    & \qquad \leq \frac{T}{1-\min_{n\geq 1} H_n} <\infty,
    \end{align*}
    and
    \begin{align*}
    &  \int_{t_j}^{t} \int_{|\xi_j|> 1} \left| \F G_{t-t_j}(\xi_1+\cdots+\xi_j) \right|^2|\xi_j|^{1-2\min_{n\geq 1} H_n}d\xi_j dt_j \\
     & \qquad \leq    \int_{t_j}^{t} \int_{\R}\exp\left( -(t-t_j)|\xi_1+\cdots+\xi_j|^2\right)  d\xi_j dt_j  \\
     & \qquad = C \int_{t_j}^t \sqrt{t-t_j}dt_j<\infty,
    \end{align*}
    which, again by iterating this computation, shows that the integral in \eqref{eq: integrand5 dim finale} is bounded also in the heat equation case. This completes the proof.
\end{proof}


\appendix

\section{Auxiliary results}
\label{appendix}

In this section, we state some results that have been applied
throughout the paper. We start with four technical lemmas, proved in
\cite{BJQ15}, which provide explicit estimates, depending on $H$,
for the norm in the space $L^2(\R;\mu^H)$ of terms involving the
Fourier transforms of the fundamental solutions of the wave and heat
equations. Finally, we will also state a tightness criterion which
will be applied in Section \ref{sec:tightness}.

We recall that, for the wave and heat equations, we have,
respectively:
\begin{equation*}
\F G_t(\xi)=\dfrac{\sin(t|\xi|)}{|\xi|} \qquad \text{and} \qquad \F
G_t(\xi)=\exp\Big(\frac{-t\xi^2}{2}\Big), \quad t>0,\, \xi\in \R.
\end{equation*}
In the following three lemmas, we will denote either one of these
two functions by $\F G_t(\xi)$. We recall that the spatial spectral
measure is given by $\mu^H(d\xi)=c_H |\xi|^{1-2H} d\xi$ .

\begin{lemma}[\cite{BJQ15}, Lemma 3.1]
    \label{lemma: 3.1}
    Let $T>0$. Then,  the integral
    $$A_T(\alpha):=\int_{0}^{T}\int_{\R} |\F G_t(\xi)|^2|\xi|^\alpha\,d\xi \,dt$$
    converges if and only if $\alpha\in(-1,1)$. In this case, it holds:
    $$A_T(\alpha)=\begin{cases}
    2^{1-\alpha}C_\alpha\dfrac{1}{2-\alpha}T^{2-\alpha}  & \text{for the wave equation,}\\
    \\
    \dfrac{2}{1-\alpha}\Gamma\Big( \dfrac{\alpha+1}{2} \Big)T^{(1-\alpha)/2} & \text{for the {heat} equation,}\\
    \end{cases}$$
    where the constant $C_\alpha$ is given by
    $${C_{\alpha}=\begin{cases}
        \dfrac{\Gamma(\alpha)}{1-\alpha}\sin(\pi\alpha/2), & \alpha\in(-1,1)\setminus\{0\},\\
        \\
        \dfrac{\pi}{2}, & \alpha=0.\\
        \end{cases}}$$
\end{lemma}

\begin{lemma}[\cite{BJQ15}, Lemma 3.4]
    \label{lemma: 3.4}
    Let $T>0$ and $\alpha\in(-1,1)$. Then, for any $h>0$, it holds:
    \begin{equation*}
    \int_{0}^{T}\int_{\R}(1-\cos(\xi h))\, |\F G_t(\xi)|^2|\xi|^\alpha\,d\xi \,dt \leq
    \begin{cases}
    C|h|^{1-\alpha} & \text{for the heat equation,} \\
    CT|h|^{1-\alpha} & \text{for the wave equation,} \\
    \end{cases}
    \end{equation*}
    where $C=\int_{\R} (1-\cos\eta)|\eta|^{\alpha-2}d\eta$.
\end{lemma}

\begin{lemma}[\cite{BJQ15}, Lemma 3.5]
    \label{lemma: 3.5}
    Let $T>0$ and $\alpha\in(-1,1)$. Then, for any $h>0$, it holds:
    \begin{equation*}
    \int_{0}^{T}\int_{\R} |\F G_{t+h}(\xi)-\F G_{t}(\xi)|^2|\xi|^\alpha\,d\xi \,dt \leq
    \begin{cases}
    C_\alpha |h|^{(1-\alpha)/2} & \text{for the heat equation,} \\
    C_{\alpha} T|h|^{1-\alpha} & \text{for the wave equation,} \\
    \end{cases}
    \end{equation*}
    where
    $$C_\alpha= \int_{\mathbb
        R}\frac{(1-e^{-\eta^2/2})^2}{|\eta|^{2-\alpha}}d\eta \quad\text{for
        the heat equation, and} $$
    $$C_\alpha=4\int_{\mathbb R}\frac{\min(1,|\eta|^2)}{|\eta|^{2-\alpha}}d\eta \quad \text{for the wave equation.}
    $$
\end{lemma}

\begin{lemma}[\cite{BJQ15}, Lemma D.2]
    \label{lemma: D.2}
    For any $H\in (0, \frac 12)$ and for any $\xi\in\R$, we have:
    $$\int_{\R} \frac{|1-e^{-i\xi x}|^2}{|x|^{2-2H}}dx=|\xi|^{1-2H}\frac{2 \Gamma(2H+1)\sin(\pi H)}{H(1-2H)}$$
\end{lemma}

\medskip

The following tightness criterion on the plane was proved in
\cite[Prop. 2.3]{Yor}.
\begin{theorem}
    \label{th: centsov}
    Let $\{X_\lambda\}_{\lambda\in\Lambda}$ be a family of random functions indexed on the
    set $\Lambda$ and taking values in the space $\C([0,T]\times \R)$, in which we consider
    the metric of uniform convergence over compact sets. Then, the family
    $\{X_\lambda\}_{\lambda\in \Lambda}$ is tight if, for any compact set $J\subset \R$,
    there exist $p',p>0$, $\delta>2$, and a constant $C$ such that the following holds
    for any $t',t\in [0,T]$ and $x',x\in J$:
    \begin{itemize}
        \item[(i)] $\sup_{\lambda\in \Lambda}\EE\left[|X_\lambda(0,0)|^{p'}\right]<\infty$,
        \item[(ii)] $\sup_{\lambda\in \Lambda} \EE\left[|X_\lambda(t',x')-X_\lambda(t,x)|^p\right]\leq
        C\left(|t'-t|+|x'-x|\right)^\delta$.
    \end{itemize}
\end{theorem}


\section*{Acknowledgement}

Research supported by the grants MTM2015-67802P and
PGC2018-097848-B-I00 (Ministerio de Econom\'ia y Competitividad).


\end{document}